\documentclass[11pt]{amsart}

\usepackage[a4paper,vmargin=4cm]{geometry} 
\usepackage{amsfonts,amssymb,amscd,amstext} 
\usepackage{graphicx} 
\usepackage{color} 
\usepackage[dvips]{epsfig} 

\usepackage[utf8]{inputenc} 
\usepackage{hyperref} 

\usepackage{verbatim} 
\usepackage{fancyhdr} 
\input xy 
\xyoption{all} 

\usepackage{times} 
\usepackage{enumerate} 
\usepackage{titlesec} 
\usepackage{mathrsfs} 

\numberwithin{equation}{section} 
\numberwithin{figure}{section}

\titleformat{\section}
{\filcenter\bfseries\large} {\thesection{.}}{0.2cm}{}
\titleformat{\subsection}[runin] 
{\bfseries} {\thesubsection{.}}{0.15cm}{}[.] 
\titleformat{\subsubsection}[runin] 
{\em}{\thesubsubsection{.}}{0.15cm}{}[.] 

\usepackage[up,bf]{caption} 

\newtheorem{theorem}{Theorem}[section] 
\newtheorem{proposition}[theorem]{Proposition}

\newtheorem{corollary}[theorem]{Corollary} 

\theoremstyle{definition} 
\newtheorem{definition}[theorem]{Definition} 
\newtheorem{remark}[theorem]{Remark} 
 
\newtheorem{problem}[theorem]{Problem} 
\newtheorem{example}[theorem]{Example} 
 
\newcommand\Ical{\mathcal{I}}

\newcommand\Cscr{\mathscr{C}}

\newcommand\Oscr{\mathscr{O}}

\newcommand\B{\mathbb{B}} 
\newcommand\C{\mathbb{C}} 
 
\newcommand\CP{\mathbb{CP}}

\newcommand\N{\mathbb{N}}

\newcommand\U{\mathbb{U}} 
\newcommand\Z{\mathbb{Z}}

\newcommand\igot{\mathfrak{i}}

\renewcommand\igot{\mathfrak{i}}

\newcommand\E{\mathrm{e}} 
 
\renewcommand\imath{\igot}

\newcommand\hra{\hookrightarrow}

\newcommand\wt{\widetilde} 
\newcommand\wh{\widehat} 
\newcommand\di{\partial} 
\newcommand\dibar{\overline\partial} 
\newcommand{\nad}[2]{\genfrac{}{}{0pt}{}{#1}{#2}}

\newcommand\dist{\mathrm{dist}} 
\renewcommand\span{\mathrm{span}}

\newcommand\Aut{\mathrm{Aut}} 
 
\newcommand\Pic{\mathrm{Pic}}

\def\dist{\mathrm{dist}} 
\def\span{\mathrm{span}} 
 
\def\rank{\mathrm{rank}} 
 
\def\Ell1{\mathrm{Ell_1}} 
\def\CEll1{\mathrm{CEll_1}}

\begin{document} 

\fancyhead[LO]{Oka tubes in holomorphic line bundles} 
\fancyhead[RE]{F.\ Forstneri\v c and Y.\ Kusakabe} 
\fancyhead[RO,LE]{\thepage} 

\thispagestyle{empty} 


\begin{center} 
{\bf \LARGE Oka tubes in holomorphic line bundles} 

\vspace*{0.4cm} 

{\large\bf Franc Forstneri\v c and Yuta Kusakabe} 
\end{center} 

\vspace*{0.5cm} 

{\small 
\noindent {\bf Abstract} \ \ 
Let $(E,h)$ be a semipositive hermitian holomorphic 
line bundle on a compact complex manifold $X$ with $\dim X>1$. 
Assume that for each point $x\in X$ there exists a divisor $D\in |E|$ 
in the complete linear system determined by $E$ 
whose complement $X\setminus D$ is a Stein neighbourhood of $x$ 
with the density property. Then, the disc bundle $\Delta_h(E)=\{e\in E:|e|_h<1\}$ 
is an Oka manifold while $D_h(E)=\{e\in E:|e|_h>1\}$ 
is a Kobayashi hyperbolic domain. In particular, the zero section of $E$ 
admits a basis of Oka neighbourhoods $\{|e|_h<c\}$ with $c>0$. 
We show that this holds if $X$ is a rational homogeneous manifold of 
dimension $>1$. This class of manifolds includes complex projective 
spaces, Grassmannians, and flag manifolds.
This phenomenon contributes to the heuristic principle that Oka 
properties are related to metric positivity of complex manifolds.
\hspace*{0.1cm} 
}

\noindent{\bf Keywords:}\hspace*{0.1cm} 
Oka manifold, holomorphic line bundle, hermitian metric, 
polarised manifold 

\noindent{\bf MSC (2020):}\hspace*{0.1cm} 
Primary 32Q56; Secondary 32E10, 32L05, 32Q10 


\noindent{\bf Date:}\hspace*{0.1cm} 
30 August 2024; minor edits on 21 November 2024


%
%

\section{Introduction}\label{sec:intro} 
A complex manifold $Y$ is called an {\em Oka manifold} if holomorphic maps 
$S\to Y$ from any Stein manifold $S$ satisfy the Oka principle 
with approximation on compact holomorphically convex subsets of $S$ 
and interpolation on closed complex subvarieties of $S$; 
see \cite[Definition 5.4.1 and Theorem 5.4.4]{Forstneric2017E}. 
This is a central holomorphic flexibility notion in complex geometry, 
and it is of major interest to find new examples of Oka manifolds. 
A complex manifold $Y$ is an {\em Oka-1 manifold} 
\cite{AlarconForstneric2023Oka1} if these properties hold for maps 
$S\to Y$ from any open Riemann surface $S$. 
Every complex homogeneous manifold is Oka (see Grauert 
\cite{Grauert1957I} and \cite[Proposition 5.6.1]{Forstneric2017E}). 
Many further examples were given by Gromov \cite{Gromov1989} and others;
see the surveys in \cite{Forstneric2017E,Forstneric2023Indag}.

In this paper, we describe a new phenomenon in Oka theory, 
relating the Oka property of tubes in hermitian holomorphic line bundles
on compact Oka manifolds to the curvature properties of the metric.
We show in particular that disc bundles in many Griffiths semipositive 
holomorphic line bundles are Oka manifolds. This holds for 
semipositive ample line bundles on projective spaces 
(see Theorem \ref{th:projective}),
Grassmannians (see Proposition \ref{prop:Grassmannian}),  
their products (see Corollary \ref{cor:productofGrassmannians}),
and on any rational homogeneous manifold of 
dimension $>1$ (see Theorem \ref{th:rational}).
Our main result, Theorem \ref{th:density2}, establishes this
phenomenon for any polarised manifold $(X,E)$ 
with the polarised density property, see Definition \ref{def:PDP}. 
An important ingredient in the proofs are the recent results
of the second named author \cite{Kusakabe2024},
who found large classes of Oka manifolds given as complements 
of closed holomorphically convex sets in Stein manifolds
with the density property.

Let $\pi:E\to X$ be a holomorphic line bundle on a connected compact 
complex manifold $X$, and let $h$ be a hermitian metric on $E$. 
Denote by $|e|_h$ the norm of $e\in E$. We are interested in conditions
on $X$ and the hermitian line bundle $(E,h)$ which ensure that the disc bundle 
\begin{equation}\label{eq:db}
	\Delta_h(E)= \{e\in E: |e|_h<1\}
\end{equation}
is an Oka manifold. In particular, when does the zero section 
$E(0)=\{e\in E:|e|_h=0\}$ admit a basis of open Oka neighbourhoods? 
It turns out that these questions are related to semipositivity of the metric $h$,
hence to the existence of nontrivial holomorphic sections $X\to E$.

We begin with some immediate observations. The total space $E$ is Oka
if and only if the base $X$ is Oka \cite[Theorem 5.6.5]{Forstneric2017E}.  
Since $\Delta_h(E)$ admits a holomorphic retraction onto the zero section 
$E(0)\cong X$, if $\Delta_h(E)$ is Oka then $X$ is Oka 
\cite[Proposition 5.6.8]{Forstneric2017E}. 
For any $c>0$ the disc bundle $\Delta_{h,c}(E)=\{|e|_h<c\}$ is biholomorphic to 
$\Delta_h(E)$ by a dilation in the fibres, so an affirmative
answer to the first question implies the same for the second one.
The answers to both questions are negative for any hermitian metric $h$
on the trivial line bundle $E=X\times \C$. 
In this case, $\Delta_h(E)$ is contained in $X\times c\Delta$
for some $c>0$, where $\Delta\subset\C$ is the unit disc.
This manifold admits a bounded plurisubharmonic function coming from 
$c\Delta$ which is nonconstant on every open subset,
so it cannot contain any Oka domain \cite[Proposition 7.1.9]{Forstneric2017E}. 
The same argument applies to trivial vector bundles of higher rank on a compact
complex manifold.

A more subtle analysis is tied to the curvature of the metric $h$, which determines the geometric shape of the disc bundle \eqref{eq:db}.  
The curvature of $h$ is the $(1,1)$-form on $X$ given by
\[ 
	\imath \Theta_h=-\imath\, \di\dibar \log h=-\frac12 dd^c \log h,
	\quad  \imath=\sqrt{-1} 
\]  
(see \eqref{eq:curvature}). A hermitian holomorphic line bundle $(E,h)$
is {\em positive} if $\imath \Theta_h$ is a positive $(1,1)$-form, and 
{\em semipositive} if $\imath \Theta_h\ge 0$. A holomorphic line bundle
$E$ is positive if it admits a hermitian metric with positive curvature.
The disc bundle $\Delta_h(E)$ is a Hartogs domain in $E$, 
and the Levi form of its boundary is the 
hermitian form determined by $dd^c \log h$ (see Proposition \ref{prop:psc}).  
Hence, the metric negativity $\imath \Theta_h<0$
at $x_0\in X$ is equivalent to $\Delta_h(E)$ being strongly
pseudoconvex over a neighbourhood of $x_0$, so it is
not Oka. (Indeed, a domain with a strongly pseudoconvex
boundary point admits a nonconstant bounded plurisubharmonic 
function, hence it cannot be Oka; see \cite[Proposition 7.1.9]{Forstneric2017E}.)
If on the other hand $\imath \Theta_h\ge 0$ then $\Delta_h(E)$ is pseudoconcave, 
and we will show that it is an Oka manifold in many cases of interest. 

We begin by considering line bundles on the simplest compact
Oka manifolds, the projective spaces $\CP^n$. 
The isomorphism classes of holomorphic line bundles on a 
complex space $X$ are in bijective correspondence with the 
elements of the Picard group $\Pic(X)=H^1(X,\Oscr^*)$.
For projective spaces, $\Pic(\CP^n)\cong\Z$ is a free cyclic 
group generated by the {\em hyperplane section bundle} $\Oscr_{\CP^n}(1)$ 
(see Griffiths and Harris \cite{GriffithsHarris1994} or Wells \cite{Wells2008}). 
It is customary to write $\Oscr_{\CP^n}(k)$ for the $k$-th tensor power 
of $\Oscr_{\CP^n}(1)$. The dual $\U=\Oscr_{\CP^n}(-1)$ of $\Oscr_{\CP^n}(1)$ 
is the {\em universal bundle};
see \cite[p.\ 17, Example 2.6]{Wells2008}. The line bundle $\Oscr_{\CP^n}(k)$ 
is {\em positive} resp.\ {\em negative} according to whether $k>0$ or $k<0$. 
It admits a hermitian metric whose curvature is $k$-times the Fubini--Study form 
on $\CP^n$ (see Example \ref{ex:FB}).

%
%
\begin{theorem}\label{th:projective} 
Given a positive holomorphic line bundle $E=\Oscr_{\CP^n}(k)$ 
on $\CP^n$ $(n\ge 1,\ k\ge 1)$ and a semipositive hermitian metric $h$ on $E$ 
(i.e., $\imath \Theta_h\ge 0$), the following assertions hold. 
\begin{enumerate}[\rm (a)] 
\item The punctured disc bundle $\Delta^*_h(E)=\{e\in E: 0<|e|_h<1\}$ 
is an Oka manifold, and the disc bundle $\Delta_h(E)=\{e\in E: |e|_h<1\}$ is 
an Oka-1 manifold.
\item If $n\ge 2$ or $E=\Oscr_{\CP^n}(1)$ then the disc bundle 
$\Delta_h(E)$ is an Oka manifold. 
\item The domain $D_h(E)=E\setminus \overline{\Delta_h(E)}=\{e\in E: |e|_h>1\}$ is 
Kobayashi hyperbolic and has pseudoconvex boundary $bD_h(E)=\{|e|_h=1\}$. 
\end{enumerate} 
For a negative holomorphic line bundle $E=\Oscr_{\CP^n}(k)$ 
$(n\ge 1,\ k\le -1)$ and a seminegative hermitian metric $h$ on $E$ 
($\imath \Theta_h\le 0$), the following assertions hold. 
\begin{enumerate}[\rm (a')] 
\item The domain $\Delta^*_h(E)$ is 
Kobayashi hyperbolic and pseudoconvex along $\{|e|_h=1\}$.
\item The domain $D_h(E)=E\setminus \overline{\Delta_h(E)}$ is Oka.
\end{enumerate} 
These results hold if the metric $h$ is continuous and semipositive 
(resp.\ seminegative) in the weak sense. 
They also hold for the restrictions of these bundles to any affine Euclidean 
chart in $\CP^n$.
\end{theorem} 

With $(E,h)$ as in part (b) of the theorem, the circle bundle $\{e\in E:|e|_h=1\}$ 
splits $E$ into a relatively compact Oka domain $\{|e|_h<1\}$ and a 
hyperbolic domain $\{|e|_h>1\}$. A phenomenon of this type was first observed 
by Forstneri\v c and Wold \cite{ForstnericWold2024IMRN} who showed that, 
under a mild assumption on an unbounded closed convex set 
$K \subset \C^n$ $(n>1)$, its interior $\mathring K$ is Kobayashi hyperbolic 
while its complement $\C^n \setminus K$ is an Oka domain. 

Note that the natural projection $\Delta_h(E)\to \CP^n$ 
in Theorem \ref{th:projective} is a holomorphic submersion and 
a topological fibre bundle, its base and total space are Oka manifolds
in case (b), yet its fibres are Kobayashi hyperbolic. 
In particular, it is not an Oka map (see Definition \ref{def:Okamap})
since the fibres of an Oka map are Oka manifolds 
(see \cite[Proposition 3.14]{Forstneric2023Indag}). 
We now show that this phenomenon does not occur in holomorphic fibre bundles. 

%
%
\begin{proposition}\label{prop:hyperbolicfibre} 
If $E\to X$ is a holomorphic fibre bundle on a connected
complex manifold $X$  whose fibre $Y$ is 
Kobayashi hyperbolic with $\dim Y>0$, then $E$ is not an Oka manifold. 
\end{proposition} 

\begin{proof} 
Let $\pi:\wt X\to X$ be the universal covering. 
The pullback bundle $\pi^*E\to \wt X$ has the same fibre $Y$. 
Since $Y$ is hyperbolic, there are no nontrivial holomorphic maps 
to its holomorphic automorphism group $\Aut(Y)$ (see Kobayashi 
\cite[Theorem 5.4.5]{Kobayashi1998}), so this is a 
flat bundle. Since $\wt X$ is simply connected, 
it follows that the bundle $\pi^*E\to \wt X$ is trivial, 
isomorphic to $\wt X\times Y$ 
(see Royden \cite[Corollary 1]{Royden1974}). 
This manifold is not Oka due to the hyperbolic factor $Y$. 
Since the natural map $\pi^*E \to E$ is a holomorphic covering map 
and the class of Oka manifolds is invariant under such maps 
(see \cite[Proposition 5.6.3]{Forstneric2017E}), $E$ is not Oka. 
\end{proof} 

%
%
Theorem \ref{th:projective} is proved in Section \ref{sec:proofs}; 
here is an outline. 
If $E=\Oscr_{\CP^n}(k)$ with $k>0$, then for any hermitian metric $h$ on $E$ 
the restriction of the disc bundle $\Delta_h(E)$ to any affine Euclidean chart in 
$\CP^n$ is a Hartogs domain $\Omega$ in $\C^{n+1}$ whose radius grows
at least linearly (see Example \ref{ex:FB}). If $h$ is semipositive then 
$\Omega$ is pseudoconcave (see Proposition \ref{prop:psc} (iii')). 
By Proposition \ref{prop:Hartogs}, such a domain is Oka if $n\ge 2$. 
This result and the localization theorem for Oka manifolds 
\cite[Theorem 1.4]{Kusakabe2021IUMJ} are the key to the proof 
of part (b). Part (a) is seen by an explicit analysis
of the hyperplane section bundle $\Oscr_{\CP^n}(1)$, 
using that complements of compact polynomially convex sets in $\C^{n+1}$ 
$(n\in\N)$ are Oka (see Kusakabe \cite[Corollary 1.3]{Kusakabe2024}),
and that the relevant properties of these tubes are preserved under 
tensor powers (see Proposition \ref{prop:tensor} and Corollary \ref{cor:tensor}).
When passing to the hermitian dual bundle $(E^*,h^*)$, positivity and negativity 
get reversed and the punctured disc bundle $\Delta^*_h(E)$  
is biholomorphic to the outer tube $D_{h^*}(E^*)=\{h^*>1\}$ of the dual bundle, 
which gives part (b'). Parts (c) and (a')
follow from Grauert's result on blowing down exceptional varieties
\cite[Satz 1, p.\ 341]{Grauert1962}; see Remark \ref{rem:blowdown}. 

%
%
We now proceed towards our main results.
Recall that a holomorphic vector field on a complex manifold $X$ is said to be 
{\em complete} if its flow exists for all complex values of time, 
so it forms a complex one-parameter group of holomorphic automorphisms 
of $X$. The following class of complex manifolds was introduced by 
Varolin \cite{Varolin2001};
see also \cite[Definition 4.10.1]{Forstneric2017E}. 

%
%
\begin{definition}\label{def:density}
A complex manifold $X$ has the {\em density property} if every holomorphic 
vector field on $X$ can be approximated uniformly on compacts by sums and 
commutators of complete holomorphic vector fields on $X$.
\end{definition}

Every Stein manifold $X$ with the density property
has infinite dimensional automorphism group (hence $\dim X>1$), and 
it is an elliptic Oka manifold (see \cite[Proposition 5.6.23]{Forstneric2017E}).
The fact that the Euclidean spaces $\C^n$, $n>1$, have the density property 
was discovered by Anders\'en and Lempert \cite{AndersenLempert1992}.
Most complex Lie groups and complex homogeneous manifolds have 
the density property. Surveys can be found in \cite[Chapter 4]{Forstneric2017E}, 
\cite{ForstnericKutzschebauch2022}, and \cite{Kutzschebauch2020}.

%
%
\begin{theorem}\label{th:density1}
Assume that $X$ is a compact complex submanifold of $\CP^n$ such that 
for the affine charts $U_i\cong\C^n$ covering $\CP^n$
the Stein manifold $X\cap U_i$ has the density property for every $i=0,\ldots,n$.
Let $E\cong \Oscr_{\CP^n}(k)$ with $k\ge 1$ 
be a positive holomorphic line bundle on $\CP^n$ endowed with a hermitian 
metric $h$ satisfying $\imath\Theta_h|_{TX} \ge 0$. 
Then the disc bundle
$
	\Delta_h(E)|_X=\{e\in E|_X: |e|_h<1\}
$
is an Oka manifold while $D_h(E)|_X=\{e\in E|_X: |e|_h>1\}$ is 
a Kobayashi hyperbolic domain in $E|_X$
with pseudoconvex boundary $\{e\in E|_X : |e|_h=1\}$. 
\end{theorem}

An example satisfying Theorem \ref{th:density1} is the hyperquadric 
\begin{equation}\label{eq:hyperquadric}
	X=\bigl\{[z_0:z_1:\cdots:z_n] \in\CP^n : 
	z_0^2+z_1^2+\cdots+z_n^2=0\bigr\}, \ \ n\ge 3.
\end{equation}
The intersection of $X$ with any affine chart 
$U_i=\{z_i\neq 0\}$ is the complexified sphere
in $\C^n$, which is a Danielewski manifold and has the density property 
(see Kaliman and Kutzschebauch \cite{KalimanKutzschebauch2008MZ}).
This {\em null quadric} plays a major role in the theory of minimal
surfaces; see \cite{AlarconForstnericLopez2021}.
Another example is the Pl\"ucker embedding of a Grassmannian
of dimension $>1$; see Example \ref{ex:Grassmannian}. 
 
Denote by $|E|$ the complete linear system of divisors on $X$ associated to 
a holomorphic line bundle $E\to X$ (see e.g.\ \cite{GriffithsHarris1994}). 
The divisors in $|E|$ are the zero sets with multiplicities of nontrivial 
holomorphic sections of $E$. We shall often identify 
a divisor with its support, neglecting the multiplicities. 
The following is our main result. 

%
%
\begin{theorem}\label{th:density2} 
Let $E$ be a holomorphic line bundle on a compact complex manifold $X$.
Assume that for each point $x\in X$ there exists a divisor $D\in |E|$ 
whose complement $X\setminus D$ is a Stein 
neighbourhood of $x$ with the density property.
Given a semipositive hermitian metric $h$ on $E$, the disc bundle 
$\Delta_h(E)$ \eqref{eq:db} is an Oka manifold while 
$D_h(E)=E\setminus \overline{\Delta_h(E)}$ is  a Kobayashi hyperbolic 
domain with pseudoconvex boundary $bD_h(E)=\{|e|_h=1\}$. 
In particular, the zero section of $E$ admits a basis of Oka neighbourhoods 
$\Delta_{h,c}(E)=\{|e|_h<c\}$ with $c>0$. 
\end{theorem}

A holomorphic line bundle $E$ on a compact complex manifold $X$
is called {\em basepoint-free} if the intersection of the divisors in $|E|$ is empty. 
If this holds, there is a holomorphic map $\Phi:X\to \CP^n$
for some $n\in\N$ (see \eqref{eq:Phi}) such that $E$ is isomorphic to the pullback 
$\Phi^*\Oscr_{\CP^n}(1)$ of the hyperplane section bundle 
(see \cite[Theorem II.7.1]{Hartshorne1977}). Hence, $E$ 
admits a semipositive hermitian metric obtained by pulling back a positive 
metric on $\Oscr_{\CP^n}(1)$ (see Example \ref{ex:FB}).
This gives the following metric-free corollary to Theorem \ref{th:density2}.

\begin{corollary}\label{cor:main}
Let $E$ be a holomorphic line bundle on a compact complex manifold $X$.
If for each point $x\in X$ there exists a divisor $D\in|E|$ whose complement is 
a Stein neighbourhood of $x$ with the density property, 
then the zero section of $E$ admits a basis of Oka neighbourhoods.
\end{corollary}

Theorems \ref{th:density1} and \ref{th:density2} are proved in 
Section \ref{sec:proofs}. In Section \ref{sec:examples} we give 
several examples. To this end, we recall the following notions.
A holomorphic line bundle $E$ on a compact complex manifold $X$
is called {\em ample} if some positive tensor power $E^{\otimes k}$ 
is {\em very ample}, meaning that holomorphic sections of $E^{\otimes k}$ 
provide an embedding of $X$ into a projective space.
The Kodaira embedding theorem \cite{Kodaira1954} 
implies that a positive line bundle is ample. Conversely, every ample line bundle 
admits a hermitian metric that makes it a positive line bundle.
A {\em polarised manifold} is a pair $(X,E)$ of a 
compact complex manifold $X$ 
and an ample line bundle $E$ on $X$. Note that such $X$
is necessarily projective, and every projective manifold admits
an ample line bundle.

%
%
\begin{definition}\label{def:PDP} 
\begin{enumerate}[\rm (a)] 
\item A polarised manifold $(X,E)$ has the {\em polarised density property} if 
for each point $x\in X$ there exists a divisor $D\in |E|$ whose complement 
$X\setminus D$ is a Stein neighbourhood of $x$ with the density property.
\item 
A compact projective manifold $X$ has the polarised density property
if $(X,E)$ has the polarised density property for every ample line bundle $E$ 
on $X$.
\end{enumerate} 
\end{definition}

It is easily seen that for a polarised manifold $(X,E)$ 
and a divisor $D\in |E|$, the complement $X\setminus D$ is an affine manifold, 
hence Stein. A Stein manifold with the density property is an Oka
manifold (see \cite[Proposition 5.6.23]{Forstneric2017E}). 
Hence, if $(X,E)$ has the polarised density property, then 
$X$ is an Oka manifold by the localization theorem
\cite[Theorem 1.4]{Kusakabe2021IUMJ}.

By Proposition \ref{prop:positive}, every holomorphic line bundle satisfying 
the condition of Theorem \ref{th:density2} is ample.
Hence, Theorem \ref{th:density2} can equivalently be stated as follows.

%
%
\begin{theorem}\label{th:polarised}
If $(X,E)$ is a polarised manifold with the polarised density property,
then for any semipositive hermitian metric $h$ on $E$ the disc bundle 
$\Delta_h(E)$ \eqref{eq:db} is an Oka manifold while the domain 
$D_h(E)=E\setminus \overline{\Delta_h(E)}$ is Kobayashi hyperbolic.  
\end{theorem}

Theorem \ref{th:projective} says that the projective space $\CP^n$
of dimension $n>1$ has the polarised density property.
In Section \ref{sec:examples} we prove the following further 
results on this topic.

\begin{itemize}
\item If $(X,E)$ has the polarised density property then so 
does $(X,E^{\otimes k})$ for every $k>1$ (see Proposition \ref{prop:tensor2}). 
\item
Every complex Grassmannian (or a product of Grassmannians)
of dimension $>1$ has the polarised density property 
(see Proposition \ref{prop:Grassmannian}
and Corollary \ref{cor:productofGrassmannians}). 
\item 
If the polarised manifolds $(X_1,E_1)$ and $(X_2,E_2)$ 
have the polarised density property then 
so does their exterior tensor product $(X_1\times X_2,E_1\boxtimes E_2)$
(see Proposition \ref{prop:product}). 
\item 
If $(X,E)$ has the polarised density property, then 
$(X\times\CP^n, E\boxtimes\Oscr_{\CP^n}(k))$ $(n,k>0)$ also has the 
polarised density property (see Proposition \ref{prop:productCPr}).
\item Recall that a {\em rational manifold} is a projective manifold birationally 
isomorphic to a projective space. 
If $X_1,\ldots,X_m$ $(m\ge 2)$ are rational manifolds such that every $X_i$ 
with $\dim X_i>1$ has the polarised density property, then 
their product $X=X_1\times X_2\times \cdots\times X_m$ also has
the polarised density property (see Proposition \ref{prop:product-rational}).
\item Every rational homogeneous manifold of dimension $>1$ has the polarised density property (see Theorem \ref{th:rational}).
\end{itemize}

%
%
So far we have only discussed line bundles. 
One may ask what can be said about the Oka properties
of (semi) positive hermitian vector bundles $(E,h)$ of rank $>1$ on an Oka 
manifold $X$. In particular, when is the tube $\{e\in E:|e|_h<1\}$ Oka? 
Its boundary $\{|e|_h=1\}$ is strongly pseudoconvex in 
the fibre direction and pseudoconcave in the remaining 
directions; see \cite[Proposition 6.2]{DrinovecForstneric2010AJM}. 
We are not aware of any example of an Oka domain 
whose boundary fails to be pseudoconcave. For the same reason, 
we do not know anything about these questions if the hermitian metric 
has mixed signature.
%
%
On the other hand, we obtain the following analogue of Theorem 
\ref{th:projective} (b') for any Griffiths seminegative hermitian vector bundle
(see Griffiths \cite{Griffiths1965,Griffiths1969} and Definition 
\ref{def:Griffiths-positive}) of rank $>1$, possibly trivial, on an Oka manifold. 

%
%
\begin{theorem}\label{th:negative} 
If $(E,h)$ is a Griffiths seminegative hermitian holomorphic vector bundle of rank 
$>1$ on a (not necessarily compact) Oka manifold $X$, then 
$D_h(E)=\{e\in E:|e|_h>1\}$ is an Oka domain with pseudoconcave 
boundary $bD_h(E)=\{e\in E:|e|_h=1\}$. 
\end{theorem} 

%
%
\begin{remark}\label{rem:blowdown}
If $(E,h)$ is a Griffiths seminegative holomorphic vector bundle on a 
complex manifold $X$, then the function $\phi(e)=|e|_h^2$ 
is plurisubharmonic on $E$ (see Proposition \ref{prop:positivity-psh} and 
Remark \ref{rem:continuousGriffiths}). If in addition the metric $h$ 
is Griffiths negative then $\phi$ is strongly plurisubharmonic on $E\setminus E(0)$.
In the latter case, with $X$ compact, the zero section 
$E(0)\cong X$ is the maximal compact complex submanifold of $E$,  
which can be blown down to a point (see Grauert \cite[Satz 1, p.\ 341]{Grauert1962}). 
This gives a Stein space $\wt E$, which is typically singular 
at the blown-down point, such that the image of the tube $\{|e|_h<c\}$ is a relatively 
compact domain in $\wt E$ for any $c>0$, and the tube 
$\{0<|e|_{\tilde h}<1\}\subset E$ is Kobayashi hyperbolic for any hermitian 
metric $\tilde h$ on $E$. 
If in addition $X$ is Kobayashi hyperbolic, then every tube $\{|e|_h<c\}$ 
is also Kobayashi hyperbolic, so the zero section $E(0)$ admits 
a basis of Kobayashi hyperbolic neighbourhoods.
\end{remark}

%
%
\begin{remark}\label{rem:continuousmetric} 
The proofs of Theorems \ref{th:projective}, \ref{th:density1}, \ref{th:density2}, 
and \ref{th:negative}, given in Section \ref{sec:proofs}, show that these 
results also hold for continuous hermitian metrics. 
Indeed, the basic relationship between semipositivity or seminegativity 
of the hermitian metric and the eigenvalues of the Levi form of the 
norm function remains in place (see Remark \ref{rem:continuousGriffiths}). 
\end{remark} 

%
%
As an application of our results, we show in 
Section \ref{sec:cluster} that the Oka properties of tube domains
in holomorphic vector bundles $E\to X$ on a compact complex manifold $X$
imply the existence of holomorphic maps $S\to E$ from any Stein manifold 
$S$ with $\dim S<\dim E$ having the cluster set either in the zero section $E(0)$ 
(when $E$ is a positive line bundle; see Theorem \ref{th:cluster}) or at infinity 
(when $E$ is a Griffiths negative vector bundle; see Theorem \ref{th:proper}).

%
%
%
%
Our results contribute to the heuristic principle that Oka properties 
are related to metric positivity of complex manifolds while holomorphic rigidity 
properties, such as Kobayashi hyperbolicity, are related to metric negativity. 
Examples of this principle are discussed in \cite[Sect.\ 11]{Forstneric2023Indag};
let us recall the most important ones and mention some new ones. 

Beginning on the rigidity side, 
a hermitian manifold with holomorphic sectional curvature 
bounded above by a negative constant is Kobayashi hyperbolic; 
see Grauert and Reckziegel \cite{GrauertReckziegel1965}, 
whose result generalizes the Ahlfors--Schwarz lemma \cite{Ahlfors1938}, and 
the results by Wu and Yau \cite{WuYau2016IM,WuYau2016CAG}, 
Tosatti and Yang \cite{TosattiYang2017}, 
Diverio and Trapani \cite{DiverioTrapani2019}, 
and Broder and Stanfield \cite{BroderStanfield2023X}, among others. 
Furthermore, every compact complex manifold of Kodaira general type 
is volume hyperbolic \cite{KobayashiOchiai1975}, and hence no 
such manifold is Oka. 

On the flexibility side, it is known that 
every compact K\"ahler manifold with semipositive 
holomorphic bisectional curvature is an Oka manifold; 
see \cite[Theorem 11.4]{Forstneric2023Indag}, which follows from
the classification of such manifolds by Mori \cite{Mori1979} 
and Siu and Yau \cite{SiuYau1980} (for positive bisectional curvature, 
when they are projective spaces) and Mok \cite{Mok1988} in the 
semipositive case.  As for not necessarily K\"ahler metrics,
if $(X,h)$ is a compact connected hermitian manifold 
whose holomorphic bisectional curvature is semipositive everywhere 
and positive at a point, then $X$ is a projective space 
(see Ustinovskiy \cite[Corollary 0.3]{Ustinovskiy2019}), which is Oka. 
Every compact K\"ahler manifold with positive holomorphic sectional curvature 
is rationally connected and projective (see Yang \cite{YangX2018}).
It is conjectured that every such manifold is an Oka-1 manifold  
(see \cite[Conjecture 9.1]{AlarconForstneric2023Oka1}). 
A result of Matsumura \cite[Theorem 1.3]{Matsumura2022} 
implies that a projective manifold with semipositive holomorphic 
sectional curvature is the total space of a holomorphic fibre bundle 
over an Oka manifold with a projective rationally connected fibre enjoying the 
corresponding semipositivity. By \cite[Theorem 5.6.5]{Forstneric2017E} 
the problem whether every such manifold is Oka reduces to 
the rationally connected case. Hence, the main problem is to better understand
the relationship between (semi) positivity of holomorphic sectional curvature 
and the Oka property for rationally connected projective manifolds.

%
%
%
%
\section{Preliminaries}\label{sec:prelim} 
In this section, we recall the necessary notions and tools, and we prepare some results which will be used in the proofs given in the following section.

A holomorphic line bundle $E\to X$ is given on some open covering 
$\{U_i\}_i$ of $X$ by a 1-cocycle of nonvanishing holomorphic functions 
$\phi_{i,j}:U_{i,j}=U_i\cap U_j \to\C^*$. A point $(x,t)\in U_j\times \C$ 
with $x\in U_{i,j}$ is identified in $E$ with $(x,\phi_{i,j}(x)t)\in U_i\times\C$. 
A holomorphic section $f:X\to E$ is given by a 1-cochain
$f_i\in \Oscr(U_i)$ satisfying $f_i=\phi_{i,j}f_j$ on $U_{i,j}$.
A hermitian metric $h$ on $E$ is given 
on any holomorphic line bundle chart $(x,t)\in U_i\times\C$ by 
$h(x,t)=h_i(x)|t|^2$, where the positive functions $h_i:U_i\to (0,+\infty)$ 
satisfy the compatibility conditions 
\begin{equation}\label{eq:transition-h} 
	h_i(x)|\phi_{i,j}(x)|^2 = h_j(x)\quad \text{for $x\in U_{i,j}$}. 
\end{equation} 
The curvature of the metric $h$ is the $(1,1)$-form on $X$ given on each chart $U_i$ by 
\begin{equation}\label{eq:Thetah} 
	\Theta_h =-\di\dibar \log h_i=-\di\dibar \log h = \frac{\imath}{2} dd^c \log h. 
\end{equation} 
(The second equality holds since $\di\dibar \log|t|^2=0$ on $t\ne 0$.) 
The bundle $(E,h)$ is said to be positive (resp.\ negative) if the real $(1,1)$-form 
\begin{equation}\label{eq:curvature} 
	\imath \Theta_h=-\imath\, \di\dibar \log h=-\frac12 dd^c \log h 
\end{equation} 
on $X$ is positive (resp.\ negative). Similarly we define semipositivity and seminegativity. 
It is obvious that the restriction of a (semi) positive line bundle $E\to X$
to a complex submanifold $Y\subset X$ is (semi) positive. 
%
%
If $(E',h')$ is another hermitian holomorphic line bundle on $X$ given on the same 
open covering $\{U_i\}_i$ by the $1$-cocyle $\phi'_{i,j}$, then the tensor product 
line bundle $E\otimes E'$ is given by the $1$-cocycle 
$\phi_{i,j}\phi'_{i,j} \in \Oscr(U_{i,j},\C^*)$. If $f$ and $f'$ are holomorphic 
sections of $E$ and $E'$, respectively, given by 1-cochains $f_i, f'_i\in \Oscr(U_i)$,
then $f\otimes f'$ is a holomorphic section of $E\otimes E'$
given by the 1-cochain $f_if'_i\in \Oscr(U_i)$. If a hermitian metric
$h'$ on $E'$ is given by functions $h'_i:U_i\to(0,\infty)$, then the product metric 
$h\otimes h'$ on $E\otimes E'$ is defined by the collection 
$h_ih'_i:U_i\to (0,\infty)$. From \eqref{eq:Thetah} we see that  
\[
	\Theta_{h\otimes h'} = \Theta_h + \Theta_{h'}. 
\]
Hence, the product of semipositive metrics is semipositive,
and is positive if one of the metrics is positive. For $k\in\Z$ we denote by 
$E^{\otimes k}$ the $k$-th tensor power of $E$, given by the $1$-cocycle 
$\phi_{i,j}^k$. If $h$ is a hermitian metric on $E$ given by functions $h_i(x)$
\eqref{eq:transition-h}, then the metric $h^{\otimes k}$ on $E^{\otimes k}$ 
is given by the functions $h_i(x)^k$ for $x\in U_{i}$. 
The hermitian dual bundle $(E^*,h^*)$ is naturally isomorphic
to $(E^{-1},h^{-1})$, where we omitted the tensor product sign.
From \eqref{eq:Thetah} we see that 
\[
	\Theta_{h^{\otimes k}} = k \, \Theta_h\quad \text{for all $k\in\Z$}. 
\]
Conversely, if $E=L^{\otimes k}$ $(k\ne 0)$ and $h$ is a hermitian metric on $E$ given in 
charts $U_i\subset X$ by positive functions $h_i$, then 
$h=\tilde h^{\otimes k}$ where $\tilde h$ is a hermitian metric on the line bundle 
$L$ defined by the collection of functions $\tilde h_i=h_i^{1/k}:U_i\to (0,\infty)$.

%
%
\begin{proposition}\label{prop:tensor} 
Let $(E,h)$ be a hermitian holomorphic line bundle on a complex manifold $X$. 
\begin{enumerate} [\rm (i)] 
\item 
For every $k\in\N$ there is a surjective fibre preserving holomorphic map 
$\Psi_k:E\to E^{\otimes k}$ such that $\Psi_k(E(0))=E^{\otimes k}(0)$ and the maps
$\Psi_k:\Delta_h^*(E)\to \Delta_{h^{\otimes k}}^*(E^{\otimes k})$ 
and $\Psi_k:D_h(E) \to D_{h^{\otimes k}}(E^{\otimes k})$ 
are unbranched $k$-sheeted holomorphic coverings. 
\item
The punctured disc bundle $\Delta^*_{h}(E)$ is fibrewise biholomorphic 
to the outer tube $D_{h^{*}}(E^*)=\{h^{*}>1\}$ in the dual bundle $(E^*,h^*)$.
\end{enumerate} 
\end{proposition} 

\begin{proof} 
If $E\to X$ is given by a $1$-cocyle $\phi_{i,j}\in \Oscr^*(U_{i,j})$, then 
$E^{\otimes k}$ is given by the $1$-cocyle $\phi_{i,j}^k$. 
Denote by $\Phi_{i,j}(x,t_j) = (x,\phi_{i,j}(x)t_j)$ 
the transition maps in $E$ and by $\Phi_{i,j}^k(x,t_j) = (x,\phi_{i,j}(x)^k t_j)$ 
the associated transition maps in $E^{\otimes k}$. 
We define the map $\Psi_k$ on any chart 
$U_i\times\C$ by $\Psi_k(x,t_i)=(x,t_i^k)$. 
Since $t_i=\phi_{i,j}(x)t_j$ for $x\in U_{i,j}$, we have that 
\[ 
	(\Psi_k \circ\Phi_{i,j})(x,t_j) = \Psi_k(x,\phi_{i,j}(x)t_j) 
	= (x, \phi_{i,j}(x)^k t_j^k) = (\Phi_{i,j}^k\circ \Psi_k)(x,t_j), 
\] 
showing that $\Psi_k:E\to E^{\otimes k}$ is a well-defined 
$k$-sheeted covering projection which is branched 
along $E(0)$, and 
$\Psi_k:E\setminus E(0)\to E^{\otimes k}\setminus E^{\otimes k}(0)$ 
is an unbranched $k$-sheeted covering. From the definition of the 
metric $h^{\otimes k}$ on $E^{\otimes k}$ it follows that 
$\Psi_k:\Delta_h^*(E)\to \Delta_{h^{\otimes k}}^*(E^{\otimes k})$ 
and $\Psi_k:D_h(E) \to D_{h^{\otimes k}}(E^{\otimes k})$ 
are unbranched holomorphic coverings. This proves (i). 

%
%
Part (ii) is seen as follows.
Compactifying each fibre $E_x\cong\C$ $(x\in X)$ with the point at 
infinity yields a holomorphic fibre bundle $\wh E\to X$ with fibre $\CP^1$ 
having a well-defined $\infty$-section $E(\infty)\cong X$. 
Set $\wt E=\wh E\setminus E(0)\to X$. 
If $t\in \C$ is a coordinate on a fibre $E_x$ then $u=t^{-1}$ is a 
coordinate on $\wt E_x$, and the transition functions between the $u$-coordinates are $\phi_{i,j}^{-1}=1/\phi_{i,j}$. 
Hence, $(\wt E,h^{-1})$ is a hermitian holomorphic 
line bundle on $X$, with the zero section $\wt E(0)=E(\infty)$,  
which is naturally isomorphic to the dual line bundle $(E^*,h^*)$. 
Under this identification, the identity map 
on $\wh E$ induces a fibre preserving biholomorphism 
\begin{equation}\label{eq:Inv} 
	\Ical: E\setminus E(0) \to E^*\setminus E^*(0) 
\end{equation} 
mapping $\Delta^*_{h}(E)$ onto $D_{h^*}(E^*)=\{h^*>1\}$
and $D_{h}(E)$ onto $\Delta^*_{h^*}(E^*)$.
\end{proof} 

\begin{corollary}\label{cor:tensor}
Let $(E,h)$ be a hermitian holomorphic line bundle on a complex manifold $X$. 
\begin{enumerate} [\rm (i)] 
\item If the punctured disc bundle $\Delta^*_{h^{\otimes k}}(E^{\otimes k})$ 
is Oka for some $k\in\N$ then it is Oka for all $k\in\N$, and in such case the 
disc bundle $\Delta_{h^{\otimes k}}(E^{\otimes k})$ is Oka-1 for all $k\in\N$. 
\item
$\Delta^*_{h}(E)$ is Oka (resp.\ hyperbolic) if and only if
$D_{h^{*}}(E^*)$ is Oka (resp.\ hyperbolic).
\end{enumerate} 
\end{corollary}

\begin{proof}
All claims except the second statement in part (i) follow from Proposition
\ref{prop:tensor} and the fact that both the class of Oka manifolds and 
the class of hyperbolic manifolds are invariant under covering projections.
If $\Delta_h^*(E)$ is Oka, it is the image of a strongly dominating 
holomorphic map $\C^{n+1}\to \Delta_h^*(E)$ with $n=\dim X$ 
(see \cite{Forstneric2017Indam}). Thus, the disc bundle $\Delta_h(E)$ is 
densely dominable by $\C^{n+1}$, and hence an Oka-1 manifold 
by \cite[Corollary 2.5 (b)]{AlarconForstneric2023Oka1}. 
\end{proof}

Recall that a real function $f$ of class $\Cscr^2$ on a complex manifold $X$ 
is plurisubharmonic if $dd^c f\ge 0$, and is strongly plurisubharmonic if 
$dd^c f>0$. Both conditions generalize to upper semicontinuous functions 
with values in $[-\infty,+\infty)$ 
(see Grauert and Remmert \cite{GrauertRemmert1956}).
The curvature formula \eqref{eq:curvature} for a hermitian metric $h$
leads to the following observation, which we record for reference. 
See also Proposition \ref{prop:positivity-psh} for vector bundles of higher rank.

%
%
\begin{proposition}\label{prop:psc} 
Let $h$ be a hermitian metric of class $\Cscr^2$ on a holomorphic line 
bundle $E\to X$. The following conditions are equivalent. 
\begin{enumerate}[\rm (i)] 
\item The curvature of $h$ is seminegative: $\imath \Theta_h\le 0$. 
\item The function $\log h$ is plurisubharmonic on $E$. 
\item The disc bundle $\Delta_h(E)=\{h<1\}$ is pseudoconvex along 
$b\Delta_h(E) = \{h=1\}$. 
\end{enumerate} 
Furthermore, if $\imath \Theta_h<0$ then $h$ is strongly plurisubharmonic 
on $E\setminus E(0)$. 
Likewise, the following conditions are equivalent. 
\begin{enumerate}[\rm (i')] 
\item The curvature of $h$ is semipositive: $\imath \Theta_h\ge 0$. 
\item The function $-\log h$ is plurisubharmonic on $E\setminus E(0)$. 
\item The disc bundle $\Delta_h(E)=\{h<1\}$ is pseudoconcave along 
$b\Delta_h(E) = \{h=1\}$. 
\end{enumerate} 
\end{proposition} 

\begin{proof} 
The equivalence (i) $\Leftrightarrow$ (ii) is an immediate consequence of 
the curvature formula \eqref{eq:curvature}. 
Assume now that $U$ is a Stein domain in $X$ such that 
$E|_U\cong U\times \C$ is a trivial line bundle. 
On this chart we have $h(x,t)=\xi(x)|t|^2$ for some 
positive $\Cscr^2$ function $\xi$ on $U$, and 
\begin{equation}\label{eq:DeltaHartogs} 
	\Delta_h(E)|_U = \{(x,t)\in U\times \C: h(x,t)<1\} 
	= \{(x,t) \in U\times \C: |t|^2 \E^{\log \xi(x)} <1\} 
\end{equation} 
is a Hartogs domain in $U\times \C$. If $\log h$ is plurisubharmonic 
(condition (ii) holds) then so is $h$, and hence $\Delta_h(E)|_U$ is 
pseudoconvex. The converse is also well-known and easily seen: if the 
Hartogs domain \eqref{eq:DeltaHartogs} is pseudoconvex then $\log\xi$ is plurisubharmonic on $U$, and hence $\log h$ is plurisubharmonic on $E|_U$.
This proves (ii) $\Leftrightarrow$ (iii). 
If $\imath \Theta_h<0$ then $\log \xi$ and hence $\xi$ are strongly 
plurisubharmonic, so $h$ is strongly plurisubharmonic 
on $E\setminus E(0)$. The equivalences 
(i')\ $\Leftrightarrow$\ (ii')\ $\Leftrightarrow$\ (iii') are proved in the same 
way and we leave out the details. 
\end{proof}

%
%
\begin{example}[Special hermitian line bundles on projective spaces] \label{ex:FB} 
Let $z=(z_0,z_1,\ldots,z_n)$ be Euclidean coordinates on $\C^{n+1}$ 
and $[z]=[z_0:z_1:\cdots:z_n]$ the associated homogeneous coordinates 
on $\CP^n$. On the affine chart $U_i=\{[z]\in\CP^n: z_i\ne 0\}\cong \C^n$ 
$(i=0,1,\ldots,n)$ we have the affine coordinates $z^i=(z_0/z_i,\ldots,z_n/z_i)$, 
where the term $z_i/z_i=1$ omitted. Fix $k\in\Z$ and define a hermitian metric 
$h$ on $E=\Oscr_{\CP^n}(k)$ by 
\begin{equation}\label{eq:h} 
	h([z],t)= \frac{|t|^2}{(1+|z^i|^2)^k} = \frac{|z_i|^{2k}}{|z|^{2k}} |t|^2 
	\quad \text{for $[z]\in U_i$ and $t\in\C$}. 
\end{equation} 
The transition functions on $\Oscr_{\CP^n}(k)$ are 
$\phi_{i,j}([z])=(z_j/z_i)^k$ (see \cite[p.\ 18]{Wells2008}). 
In view of \eqref{eq:transition-h} we see that $h={\tilde h}^{\otimes k}$,  
where $\tilde h$ is the metric on $\Oscr_{\CP^n}(1)$ given by \eqref{eq:h} 
with $k=1$. It follows from \eqref{eq:curvature} and \eqref{eq:h} that 
$
	\imath\Theta_h = k \, \imath\, \di\dibar \log\bigl(|z|^{2}\bigr), 
$
which is $k$-times the Fubini--Study form on $\CP^n$. 
Identifying $U_i$ with $\C^n$, the disc tube of the bundle 
$E=\Oscr_{\CP^n}(k)$ with the metric \eqref{eq:h},
restricted to $U_i$, is given by 
\begin{equation}\label{eq:discbundlek} 
	\Delta_h(E)|_{U_i} =\big\{(z,t)\in \C^n\times\C: |t|<(1+|z|^2)^{k/2}\big\}. 
\end{equation} 
This is a Hartogs domain whose radius is of order $|z|^{k}$ as $|z|\to\infty$. 
Since any two hermitian metrics on $E$ are comparable, 
the disc bundle of any hermitian metric on $E$ grows at this rate.
\end{example} 

%
%
We now recall the notions of Griffiths (semi) positivity and Griffiths 
(semi) negativity of a hermitian holomorphic vector bundle $(E,h)$ of arbitrary 
rank $r\ge 1$ on a complex manifold $X$ of dimension $n$ 
(see Griffiths \cite{Griffiths1965,Griffiths1969}). 
The hermitian metric $h$ on $E$ is given in any local frame 
$(e_1,\cdots,e_r)$ by a hermitian matrix function $h=(h_{\lambda\mu})$ with 
\[ 
	h_{\lambda\mu}(x)= 
	(e_\mu(x),e_\lambda(x))_h \quad \text{for}\ \lambda,\mu=1,\ldots,r. 
\] 
Its connection matrix $\theta_h$ and the curvature form $\Theta_h$ 
are given in any local holomorphic frame by 
\[ 
	\theta_h=h^{-1}\di h, \qquad 
	\Theta_h=\dibar \,\theta_h = 
	-h^{-1}\di\dibar h + h^{-1}\di h \wedge h^{-1}\dibar h. 
\] 
(See \cite[Chapter V]{Demailly-book} or \cite[Chapter III]{Wells2008}.) 
For a line bundle, these equal $\theta_h= h^{-1}\di h = \di\log h$ and 
$\Theta_h= - \di\dibar \log h$ (cf.\ \eqref{eq:Thetah}). 
In local holomorphic coordinates $z=(z_1,\ldots,z_n)$ on $X$ and a local 
frame $(e_1,\ldots,e_r)$ on $E$, we can identify the curvature tensor 
\[ 
	\imath\Theta_h=\sum_{\nad{i,j=1,\ldots,n}{\lambda,\mu=1,\ldots,r}} 
	c_{ij\lambda\mu} dz_i\wedge d\bar z_j \,\cdotp e^*_\lambda \otimes e_\mu 
\] 
with the hermitian form on $TX\otimes E$ given by 
\[ 
	\widetilde\Theta_h(\xi\otimes v) 
	=\sum_{\nad{i,j=1,\ldots,n}{\lambda,\mu=1,\ldots,r}} 
	c_{ij\lambda\mu} \xi_i \bar{\xi}_j v_\lambda \bar{v}_\mu. 
\] 
The following notions are due to Griffiths \cite{Griffiths1965,Griffiths1969}; 
see also \cite{AndreottiGrauert1962} and \cite[Chapter VII]{Demailly-book}. 

%
%
%
%
\begin{definition}\label{def:Griffiths-positive} 
Let $E\to X$ be a holomorphic vector bundle.  A hermitian metric $h$ on $E$ is 
{\em Griffiths semipositive} (resp.\ {\em Griffiths seminegative}) if 
$\widetilde\Theta_h(\xi\otimes v)\geq 0$ 
(resp.\ $\widetilde\Theta_h(\xi\otimes v)\leq 0$) 
for all $\xi\in T_{x}X$ and $v\in E_x$ ($x\in X$). 
If there is strict inequality for all $\xi\in T_{x}X\setminus\{0\}$ and 
$v\in E_x\setminus\{0\}$ ($x\in X$) then the metric is {\em Griffiths positive} 
(resp.\ {\em Griffiths negative}). 
\end{definition}

For line bundles, Griffiths positivity (resp.\ negativity) 
coincides with the previous definition. The following proposition explains 
the connection between Griffiths seminegativity of a hermitian metric and plurisubharmonicity of the associated squared norm function. 

%
%
\begin{proposition}\label{prop:positivity-psh} 
For a hermitian metric $h$ on a holomorphic vector bundle $E\to X$ 
the following conditions are equivalent: 
\begin{enumerate}[\rm (i)] 
\item The metric $h$ is Griffiths seminegative. 
\item For any local holomorphic section $u$ of $E$, the function $|u|_h^2$ is plurisubharmonic. 
\item For any local holomorphic section $u$ of $E$, the function $\log|u|^2_h$ is plurisubharmonic. 
\item The squared norm function $\phi(e)=|e|_h^2$ is plurisubharmonic on $E$. 
\item The function $\log\phi(e)=\log|e|_h^2$ is plurisubharmonic on $E$. 
\end{enumerate} 
If $h$ is Griffiths negative then the function $\phi$ in (iv) is strongly plurisubharmonic
on $E\setminus E(0)$.
\end{proposition} 

The equivalences between (i), (ii) and (iii) can be found in 
Raufi \cite[Sect.\ 2]{Raufi2015}. The equivalences (ii) $\Leftrightarrow$ (iv) 
and (iii) $\Leftrightarrow$ (v) are obvious. For the last statement, 
see \cite[Proposition 6.2]{DrinovecForstneric2010AJM}.

%
%
\begin{remark}\label{rem:continuousGriffiths} 
The conditions (ii) and (iii) in Proposition \ref{prop:positivity-psh} are 
equivalent also for a {\em continuous} hermitian metric on a holomorphic 
vector bundle, and they can be used to define Griffiths seminegativity and 
semipositivity for a not necessarily smooth hermitian metric 
(see \cite[Definition 1.2 and Sect.\ 2]{Raufi2015}). 
The equivalences (ii) $\Leftrightarrow$ (iv) and (iii) $\Leftrightarrow$ (v) are obvious 
in the continuous case as well. For a metric of class $\Cscr^2$, the relationship 
between the eigenvalues of the curvature form and those of 
the Levi form of the squared norm function can be found 
in Griffiths \cite[p.\ 426]{Griffiths1966}; see also the summary in 
\cite[Proposition 6.2]{DrinovecForstneric2010AJM}. 
In the notation used in the latter paper, Griffiths seminegativity can be 
expressed by $s(e)=0$ for every $e\in E\setminus E(0)$. 
\end{remark} 

%
%

In the final part of this section we recall some notions from Oka theory 
and a couple of results which are frequently used in the sequel.
We begin by recalling the notion of Oka property of a holomorphic map and of
{\em Oka map}; see \cite[Definitions 7.4.1 and 7.4.7]{Forstneric2017E}
where this is called the {\em parametric Oka property with approximation 
and interpolation}, abbreviated POPAI. 

%
%
\begin{definition}\label{def:Okamap} 
A holomorphic map $\pi: Y \to Z$ of reduced complex spaces has 
the {\em Oka property} 
if holomorphic maps $f:X\to Z$ from any Stein manifold $X$ 
satisfy the parametric $h$-principle for liftings $F:X\to Y$ with $\pi \circ F=f$.
The map $\pi:Y\to Z$ is an {\em Oka map} if it satisfies the Oka property 
and is a topological (Serre) fibration.
\end{definition}

More precisely, the Oka property of the map $\pi:Y\to Z$ means
that every continuous lifting $F_0:X\to Y$ of a given 
holomorphic map $f:X\to Z$ is homotopic through liftings of $f$
to a holomorphic lifting $F:X\to Y$. Furthermore, if $F_0$ is holomorphic 
on a compact $\Oscr(X)$-convex subset $K\subset X$ and on a closed complex
subvariety $X'\subset X$, then the homotopy of liftings $F_t:X\to Y$ $(t\in[0,1])$ 
can be chosen such that every map $F_t$ is holomorphic on 
$K\cup X'$, it agrees with $F_0$ on $X'$, and it approximates $F_0$ 
uniformly on $K$ and uniformly in the parameter $t\in[0,1]$.
The analogous conditions hold for any continuous family
of holomorphic maps $f_p:X\to Z$ depending on a parameter $p$ 
in a compact Hausdorff space. 

For a holomorphic submersion $\pi:Y\to Z$, the basic Oka property implies
the parametric Oka property (see \cite[Theorem 7.4.3]{Forstneric2017E}).
If $\pi: Y \to Z$ is an Oka map of complex manifolds with $Z$ 
connected then $\pi$ is a surjective submersion, its fibres 
are Oka manifolds (see \cite[Proposition 3.14]{Forstneric2023Indag}), 
and $Y$ is an Oka manifold if and only if $Z$ is an Oka manifold
(see \cite[Theorem 3.15]{Forstneric2023Indag}). 

The following result is due to 
Kusakabe \cite[Lemma 5.1]{Kusakabe2024}; 
see also \cite[Proposition 3.18]{Forstneric2023Indag}.

%
%
\begin{proposition} \label{prop:projections}
Assume that for every point $y$ in a complex manifold $Y$ there exist complex 
manifolds $Z_1,\ldots, Z_k$ and holomorphic submersions 
$\pi_j:Y\to Z_j$ $(j=1,\ldots,k)$ enjoying the Oka property such that $T_y Y = \sum_{j=1}^k \ker (d\pi_j)_y$.
Then $Y$ is an Oka manifold.
\end{proposition}

An unbounded closed set $S$ in a complex manifold $Y$ is called 
holomorphically convex (or $\Oscr(Y)$-convex) if $S$ is the union of 
an increasing sequence of compact $\Oscr(Y)$-convex sets.

%
%
\begin{definition}[Definition 4.1 in \cite{Kusakabe2024}]
\label{def:family} 
Let $\pi:Y\to Z$ be a holomorphic submersion. 
A closed subset $S$ of $Y$ is called a 
\emph{family of compact holomorphically convex sets} if the restriction 
$\pi|_{S}:S\to Z$ is proper and each point of $Z$ admits an open  
neighbourhood $U\subset Z$ such that the set 
$S\cap\pi^{-1}(U)$ is $\Oscr(\pi^{-1}(U))$-convex.
\end{definition}

The following is a special case of \cite[Theorem 4.2]{Kusakabe2024}
which is used in this paper. The notion of the density property was introduced
in Definition \ref{def:density}. 

%
%
\begin{theorem}\label{th:complements4.2}
Let $\pi:Y\to Z$ be a holomorphic fibre bundle whose fibre is a Stein manifold
with the density property, and let $S\subset Y$ be a family of compact 
holomorphically convex sets. Then the restriction 
$\pi|_{Y\setminus S}:Y\setminus S\to Z$ enjoys the Oka property.
\end{theorem}

In \cite[Theorem 4.2]{Kusakabe2024} it is assumed that the map $\pi:Y\to Z$ 
is a holomorphic submersion and each point of $Z$ admits an open 
neighborhood $U\subset Z$ such that $\pi^{-1}(U)$ is Stein and the 
restriction $\pi^{-1}(U)\to U$ enjoys the fibred density property. 
When $\pi:Y\to Z$ is a holomorphic fibre bundle, the latter condition 
clearly holds if the fibre is Stein and has the density property.

%
%
%
%
\section{Proofs of the main results}
\label{sec:proofs}
 
In this section, we prove Theorems \ref{th:projective}, \ref{th:density1}, 
\ref{th:density2}, and \ref{th:negative}. We also obtain Theorem \ref{th:K1.6bis}. 

\begin{proof}[Proof of Theorem \ref{th:projective}]
We begin by considering the hyperplane section bundle $\Oscr_{\CP^n}(1)$. 
Its total space $E$ can be identified with $\CP^{n+1}\setminus \{0\}$, 
where $0\in \C^{n+1}$ is an affine chart in $\CP^{n+1}$, such that the
zero section $E(0)$ is the hyperplane at infinity 
$\CP^{n+1}\setminus \C^{n+1}\cong\CP^n$ and the fibres of the projection 
$\pi:E\to \CP^n$ are the punctured complex lines through 
the origin $0\in \C^{n+1}$, with the added point at infinity. 
Let $h$ be a semipositive hermitian metric on $E$, $\imath\Theta_h\ge 0$.  
Proposition \ref{prop:psc} shows that the function $1/h$ is plurisubharmonic
on $E\setminus E(0)=\C^{n+1}\setminus \{0\}$. Clearly, this function
tends to infinity at $E(0)$ and to $0$ at the origin $0\in \C^{n+1}$,
so it extends to a plurisubharmonic exhaustion function on $\C^{n+1}$.
Therefore, the set $K=\{1/h\le 1\}=\{h\ge 1\}$ is a compact polynomially convex
neighbourhood of the origin (see Stout \cite[Theorem 1.3.11]{Stout2007}).
Note that $\Delta^*_h(E)=\C^{n+1}\setminus K$, which is Oka by 
\cite[Corollary 1.3]{Kusakabe2024}
(see also \cite[Theorem 1.2]{ForstnericWold2020MRL}), 
$\Delta_h(E)=\CP^{n+1}\setminus K$, which is Oka by 
\cite[Corollary 5.2]{Forstneric2023Indag}, and 
$D_h(E)=\mathring K\setminus\{0\}$ is a bounded domain 
in $\C^{n+1}$, hence Kobayashi hyperbolic.
If $H$ is a complex hyperplane in $\CP^{n+1}$ then 
$\CP^{n+1}\setminus (H\cup K)$ is Oka by 
\cite[Theorem 1.3]{ForstnericWold2024IMRN}. This shows that 
for any affine chart $\C^n\cong U\subset\CP^n$ the 
disc bundle $\Delta_h(E)|_U$ is Oka. 
This proves the theorem for $E=\Oscr_{\CP^n}(1)$. 

For its dual bundle $E^*=\U=\Oscr_{\CP^n}(-1)$, the universal bundle on 
$\CP^n$, parts (a') and (b') of the theorem follow immediately from the results 
for $E=\Oscr_{\CP^n}(1)$ in view of Proposition \ref{prop:tensor} (ii).
Indeed, the total space of $\U$ is biholomorphic to $\C^{n+1}$ blown up at 
the origin, its zero section $\U(0)$ is the exceptional fibre 
over $0\in\C^{n+1}$, the fibres of the projection $\pi: \U\to\CP^n$ 
are the complex lines $\C z$ for $z\in \C^{n+1}\setminus \{0\}$, 
and $\U\setminus \U(0)$ is biholomorphic to $\C^{n+1}\setminus \{0\}$. 
If $h$ is a seminegative hermitian metric on $\U$ then $1/h$ is a
semipositive hermitian metric on $E=\CP^{n+1}\setminus \{0\}$, 
$K=\{h\le 1\}$ is a compact polynomially convex neighbourhood of the origin
blown up at the origin, the domain $D_h(\U)=\{h>1\}=\Delta^*_{1/h}(E)$ is 
Oka, and the domain $\Delta_h^*(\U)=\{0<h<1\}=\{1/h>1\}=D_{1/h}(E)$
is hyperbolic. 

For the tensor powers $E^{\otimes k}=\Oscr_{\CP^n}(k)$ with $k>1$,
parts (a) and (c) follow from the already proved 
result for $E=\Oscr_{\CP^n}(1)$ by Corollary \ref{cor:tensor}. 
Likewise, the proofs of (a') and (b') for
$\U^{\otimes k}=\Oscr_{\CP^n}(-k)$ with $k>1$ follow from the
case for $\U$ by Corollary \ref{cor:tensor}.

It remains to prove part (b) for semipositive bundles $(E,h)$ with 
$E=\Oscr_{\CP^n}(k)$ when $k\ge 1$ and $n\ge 2$.
The key to the proof is the following result of independent interest.
The main idea used in this proposition will also be applied in the proofs of Theorems \ref{th:density1} and \ref{th:density2}.

%
%
\begin{proposition} \label{prop:Hartogs} 
Assume that $\phi$ is a positive continuous function on $\C^n$ $(n\ge 2)$ 
such that $\log \phi$ is plurisubharmonic, and there is a constant $c>0$ such 
that $\phi(z)\ge c\,|z|$ holds for all $z\in\C^n$. Then, the (pseudoconcave) 
Hartogs domain 
\begin{equation}\label{eq:Hartogs} 
	\Omega=\{(z,t)\in \C^n\times\C: |t|<\phi(z)\}
\end{equation} 
is an Oka domain. 
\end{proposition} 

\begin{proof} 
Let $T:\C^{n+1}=\C^n\times \C \to\C$ denote the projection $T(z,t)=t$. 
Consider the closed set 
\begin{eqnarray*} 
	S &=& \C^{n+1}\setminus \Omega =\{(z,t)\in \C^n\times \C:|t|\ge \phi(z)\} \\ 
	&=& \{(z,t)\in \C^n\times \C^*: \log\phi(z) -\log|t| \le 0\}. 
\end{eqnarray*} 
Since $\log|t|$ is harmonic on $t\in \C^*$, the function 
$\psi(z,t)=\log\phi(z) -\log|t|$ is plurisubharmonic on $\C^n\times \C^*$. 
Since $\phi$ is assumed to grow at least linearly near infinity, 
the restricted projection $T|_S:S\to\C$ is proper. It follows that for 
every $r>0$ the set 
\begin{equation}\label{eq:Kt1t2} 
	S_{r}= \{(z,t)\in S: |t|\le r\}=\{(z,t)\in \C^n\times \C^*: 
			\psi(z,t)\le 0,\ \log|t|\le \log r\} 
\end{equation} 
is compact and $\Oscr(\C^n\times \C^*)$-convex 
(see \cite[Theorem 1.3.11]{Stout2007}). 
Since the fibre of the map $T$ is $\C^n$ with $n\ge 2$, which has the 
density property, Theorem \ref{th:complements4.2} implies that the 
restricted projection $T:(\C^{n}\times \C^*)\setminus S\to\C^*$ has the 
Oka property. Since $S\cap \{t=0\}=\varnothing$, the projection 
$T:\C^{n+1}\setminus S\to\C$ has the Oka property as well 
(see \cite[Theorem 4.1]{Kusakabe2021MZ}, or use the localization principle 
for the Oka property of a holomorphic submersion, 
given by \cite[Theorem 7.4.4]{Forstneric2017E} and originally proved in 
\cite[Theorem 4.7]{Forstneric2010CA}). 

Since the function $\phi$ in \eqref{eq:Hartogs} is assumed to grow at 
least linearly at infinity, we have that $\Lambda\cap S=\varnothing$ for every 
complex hyperplane $\Lambda\subset\C^{n+1}$ sufficiently close to 
$\Lambda_0=\{t=0\}$, and there is a path $\Lambda_s$ $(s\in[0,1])$ 
of such hyperplanes connecting $\Lambda_0$ to $\Lambda$. 
For any such $\Lambda$, the set $S_r$ in \eqref{eq:Kt1t2} is also 
$\Oscr(\C^{n+1}\setminus\Lambda)$-convex by 
\cite[Corollary A.5]{Forstneric2023Indag}. 
As $r\to\infty$ these sets exhaust $S$, so $S$ is 
$\Oscr(\C^{n+1}\setminus\Lambda)$-convex. 
Let $T_\Lambda:\C^{n+1}\to\C$ be a $\C$-linear projection with 
$(T_\Lambda)^{-1}(0)=\Lambda$. 
If $\Lambda$ is sufficiently close to $\Lambda_0$ then 
the restricted projection $T_\Lambda:S\to\C$ is still proper. 
Using again Theorem \ref{th:complements4.2},  
we infer that the projection $T_\Lambda :\C^{n+1}\setminus S\to\C$ 
has the Oka property. Applying this conclusion for two linearly independent 
projections and using Proposition \ref{prop:projections}, we see that 
$\C^{n+1}\setminus S=\Omega$ is an Oka manifold. 
\end{proof} 

We continue with the proof of part (b) in Theorem \ref{th:projective}.
Let $E$ be a positive holomorphic line bundle with a semipositive 
hermitian metric $h$ on $\CP^n$ with $n>1$. 
From the equivalences (i')\ $\Leftrightarrow$\ (ii')\ $\Leftrightarrow$\ (iii')
in Proposition \ref{prop:psc} and \eqref{eq:DeltaHartogs}, 
we see that the restriction of the disc bundle 
$\Delta_h(E)$ to any affine chart $\C^{n} \cong U\subset \CP^n$ is a 
pseudoconcave Hartogs domain of the form \eqref{eq:Hartogs} with 
$\log\phi$ plurisubharmonic. We have seen in Example \ref{ex:FB} that 
the function $\phi$ grows at least linearly near infinity. 
Hence, Proposition \ref{prop:Hartogs} implies that 
$\Delta_h(E)|_U$ is an Oka domain. Note that $\Delta_h(E)|_U$ 
is a Zariski open domain in $\Delta_h(E)$. Since charts of this kind 
cover $\Delta_h(E)$, the localization theorem for Oka manifolds 
(see \cite[Theorem 1.4]{Kusakabe2021IUMJ}) implies that $\Delta_h(E)$ is Oka.  
\end{proof}

%
%
\begin{remark} \label{rem:complements1.6}
The proof of Proposition \ref{prop:Hartogs} also 
gives the following more general result related to 
\cite[Theorem 1.6]{Kusakabe2024}.
Recall that a closed subset $S$ of a Stein manifold $X$ 
is said to be $\Oscr(X)$-convex if it is exhausted by an increasing 
sequence of compact $\Oscr(X)$-convex sets. 

%
%
\begin{theorem}\label{th:K1.6bis} 
Let $S$ be a closed subset of $\C^{n} \times \C^*$ $(n\ge 2)$ 
which is $\Oscr(\C^{n} \times \C^*)$-convex.
Assume that for every complex hyperplane 
$\C^n\cong\Lambda\subset \C^{n+1}$ 
close enough to $\Lambda_0=\C^n\times \{0\}$ we have that 
$\Lambda\cap S=\varnothing$. Then, $\C^{n+1}\setminus S$ is an Oka domain. 
\end{theorem} 

The hypothesis that the condition $\Lambda\cap S=\varnothing$ for all hyperplanes $\Lambda$ close to $\Lambda_0$ is equivalent to asking that 
the projective closures of $\Lambda_0$ and $S$ do not intersect 
at infinity. For closed subsets $S$ of Euclidean spaces of dimension $\ge 3$, 
Theorem \ref{th:K1.6bis} generalizes 
\cite[Theorem 1.1]{ForstnericWold2024IMRN} 
due to Forstneri\v c and Wold. Indeed, the holomorphic convexity hypothesis 
on the set $S$ in the latter result (where it is called $E$) is strictly stronger 
than the one in Theorem \ref{th:K1.6bis}. 
However, Theorem \ref{th:K1.6bis} does not apply to subsets of $\C^2$, 
while the cited result \cite[Theorem 1.1]{ForstnericWold2024IMRN} does. 
\end{remark} 

%
%
\begin{proof}[Proof of Theorem \ref{th:density1}] 
Let $\C^n\cong U_i\subset\CP^n$ for $i=0,\ldots,n$ be affine Euclidean 
charts covering $\CP^n$ such that the Stein manifold $X_i=X\cap U_i$ 
has the density property for every $i$. 
By the localization theorem \cite[Theorem 1.4]{Kusakabe2021IUMJ},
it suffices to prove that the restricted bundle $\Delta_h(E)|_{X_i}$
is Oka for every $i$. There is a standard trivialization 
$E|_{U_i}\cong U_i\times \C$, 
and the bundle $\Delta_h(E)|_{U_i}=\{(z,t):|t|<\phi(z)\}$ is a Hartogs domain 
of the form \eqref{eq:Hartogs} with the function $\phi:U_i\to (0,\infty)$ 
growing at least linearly near infinity (see Example \ref{ex:FB}). Hence,
\[
	\Omega_i:=\Delta_h(E)|_{X_i} = \{(x,t)\in X_i\times \C: |t|<\phi(x)\}.
\]
Since $\imath\Theta_h\ge 0$ on $X$, 
$\Omega_i$ is pseudoconcave (see the equivalence
(i')\ $\Leftrightarrow$\ (iii') in Proposition \ref{prop:psc}) and
$\log \phi$ is plurisubharmonic on $X_i$. Hence, the closed set 
\[
	S_i=E|_{X_i}\setminus \Omega_i = \{(x,t)\in X_i\times \C: |t|\ge \phi(x)\}
\]
is holomorphically convex in $X_i\times \C^*$ 
(see the proof of Proposition \ref{prop:Hartogs}).
Let $T:X_i\times\C \to \C$ denote the projection onto the second factor, 
$T(x,t)=t$. These properties imply that the restricted projection 
$T|_{S_i}:{S_i}\to\C$ is proper, and ${S_i}$ is a family of compact holomorphically 
convex sets in $X_i$ with respect to $T$ (see Definition \ref{def:family}). 
Since $X_i$ is a Stein manifold with the density property,
Theorem \ref{th:complements4.2} implies that the restricted projection 
$T:(X_i\times \C)\setminus {S_i} = \Omega_i \to \C$ has the Oka property.
We now apply the same argument with tilted projections 
$T_\Lambda: U_i\times \C \to \C$ defined by affine  
hyperplanes $\Lambda\subset U_i\times \C \cong\C^{n+1}$
sufficiently close to $\Lambda_0=U_i\times \{0\}$.
Such a hyperplane $\Lambda$ is the graph of a $\C$-linear function 
$t=\xi(x)$ of $x\in U_i\cong \C^n$, and 
$\Lambda\cap (X_i\times \C)=\{(x,t):x\in X_i,\ t=\xi(x)\}$.
The fibres of the restricted projection $T_\Lambda: X_i\times \C \to \C$
are parallel translates of $\Lambda\cap (X_i\times \C)$ in the vertical 
$t$-direction, so this projection is a (trivial) holomorphic fibre bundle 
with fibre $X_i$. Since $\phi$ grows at least linearly near infinity, 
the projection $T_\Lambda:{S_i}\to\C$ is proper if $\Lambda$ is close enough to
$\Lambda_0$. For such $\Lambda$, the same argument as before shows 
that $T_\Lambda : (X_i\times \C)\setminus {S_i}=\Omega_i \to \C$ 
has the Oka property. Clearly, finitely many such projections satisfy 
the hypotheses in Proposition \ref{prop:projections}, 
and hence $\Omega_i$ is an Oka domain for every for $i=0,\ldots,n$. 

Since $\Delta_h(E)|_X=\bigcup_{i=0}^n \Omega_i$ and every
$\Omega_i$ is Zariski open in $\Delta_h(E)|_X$, 
it follows from the localization theorem \cite[Theorem 1.4]{Kusakabe2021IUMJ}  
that $\Delta_h(E)|_X$ is Oka. The fact that the exterior tube 
$D_h(E)|_X$ is Kobayashi hyperbolic is seen as in the proof of 
Theorem \ref{th:projective} (c).
\end{proof}

%
%
\begin{proof}[Proof of Theorem \ref{th:density2}] 
By the assumption, 
there are holomorphic sections $s_0,\ldots,s_n:X\to E$ such that
$X_i=\{x\in X:s_i(x)\ne 0\}$ is a Stein manifold with the density property
for every $i=0,1,\ldots,n$ and $\bigcup_{i=0}^n X_i=X$. 
Consider the holomorphic map $\Phi:X\to\CP^n$ given by
\begin{equation}\label{eq:Phi}
	\Phi(x)=[s_0(x):s_1(x):\cdots:s_n(x)]\in \CP^n, \quad x\in X.
\end{equation}
(The map $\Phi$ is well-defined since $s_i(x)$ are elements of the $1$-dimensional vector space $E_x\cong\C$ and at least one of them is 
nonzero for every $x$.) Note that $\Phi$ maps $X_i=\{s_i\ne 0\}$ to the 
complement of the standard $i$-th hyperplane 
$\CP^{n-1}\cong H_i\subset \CP^n$,
and $\Phi^{-1}(\CP^n\setminus H_i)=X_i$.
It follows that $E$ is isomorphic to $\Phi^*\Oscr_{\CP^n}(1)$, the pullback of the 
hyperplane section bundle (see \cite[Theorem II.7.1]{Hartshorne1977}).
For completeness, we include a simple argument. 
Let $\pi^*:E^*\to X$ be the dual bundle of $\pi:E\to X$,
and denote by $\langle e^*,e\rangle$ the natural pairing of elements 
$e\in E$ and $e^*\in E^*$ over the same base point $\pi(e)=\pi^*(e^*)\in X$.
Let $\U\to\CP^n$ be the universal bundle. 
We can identify $\U$ with $\C^{n+1}$ blown up at 
the origin so that the zero section $\U(0)\cong\CP^n$ is the exceptional fibre 
over $0\in\C^{n+1}$ and the fibres of the projection $\U\to\CP^n$ 
are the complex lines in $\C^{n+1}$ through the origin.
The holomorphic map $\wt \Phi:E^* \to \C^{n+1}$ given by
\[
	\wt \Phi(e^*) = \bigl(
	\langle e^*, s_i(x)\rangle\bigr)_{i=0}^n,\quad e^*\in E^*,\ x=\pi^*(e^*)\in X
\]
maps $E^*_x$ isomorphically onto the complex line in $\C^{n+1}$ determined 
by the point $\Phi(x)\in \CP^n$, so it gives a line bundle isomorphism 
$E^*\cong \Phi^*\U$. It follows that 
$E\cong \Phi^*\U^*=\Phi^* \Oscr_{\CP^n}(1)$.

The proof can now be completed as in Theorem \ref{th:density1}.
The restricted bundle $E|_{X_i}\cong X_i\times \C$ admits a trivialization 
induced via the map $\Phi$ \eqref{eq:Phi} by the standard trivialization of 
$\Oscr_{\CP^n}(1)$ over $U_i=\CP^n\setminus H_i\cong \C^n$. 
In this trivialization, $\Delta_h(E)|_{X_i}$ is a pseudoconcave Hartogs domain 
of the form \eqref{eq:Hartogs} in $X_i\times \C$.
The same argument as in the proof of Theorem \ref{th:density1}, 
using the Oka property of tilted projections
$(X_i\times \C)\setminus \Delta_h(E)|_{X_i} \to \C$ which come from linear 
projections $\C^n \times \C\to \C$ close to the standard projection onto 
the second factor, shows that $\Delta_h(E)|_{X_i}$ is Oka for every $i=0,\ldots,n$.
By the localization theorem, it follows that $\Delta_h(E)$ is Oka.
\end{proof}

%
%
\begin{proof}[Proof of Theorem \ref{th:negative}] 
Let $\pi:E\to X$ denote the vector bundle projection and set 
$S=\{e\in E:|e|_h\leq 1\}$. Assuming that $\rank\, E=r>1$ and the 
hermitian metric $h$ is Griffiths seminegative, 
we wish to prove that the exterior tube 
$D_h(E)=E\setminus S=\{e\in E:|e|_h > 1\}$ is an Oka manifold.
Condition (iv) in Proposition \ref{prop:positivity-psh} shows that the 
squared norm function $\phi(e)=|e|_h^2$ is plurisubharmonic on $E$.
Hence, for each holomorphic chart $\psi:U \to \B^n$ from an open set 
$U\subset X$ onto the unit ball $\B^n\subset\C^n$ $(n=\dim X)$ and 
each $0<\rho<1$, the compact set $\{e\in S|_U:|\psi\circ\pi(e)| \leq \rho \}$ 
is defined by plurisubharmonic functions in the Stein manifold $E|_U$, so it 
is $\Oscr(E|_U)$-convex (see Stout \cite[Theorem 1.3.11]{Stout2007}). 
Since $E\to X$ is a holomorphic vector bundle of rank $r\ge 2$,
its fibre $\C^r$ has the density property \cite{AndersenLempert1992}. 
Hence, Theorem \ref{th:complements4.2}  
implies that the projection $\pi:D_h(E)=E\setminus S\to X$ 
has the Oka property (see Definition \ref{def:Okamap}).  
Since it is also a topological fibre bundle, it is an Oka map. 
As $X$ is an Oka manifold, it follows that $D_h(E)$ is an Oka manifold  
(see \cite[Theorem 3.15]{Forstneric2023Indag} saying that, if $Y\to X$ is 
an Oka map of complex manifolds, then $Y$ is an Oka manifold 
if and only if $X$ is an Oka manifold). 
\end{proof} 

%
%
%
%
\section{Examples of line bundles satisfying Theorem \ref{th:density2}}
\label{sec:examples}

In this section, we give examples and obtain 
functorial properties of the class of polarised manifolds 
with the polarised density property (see Definition \ref{def:PDP}). 

We first show that every holomorphic line bundle satisfying 
the condition in Theorem \ref{th:density2} is ample,
and hence it is natural to restrict ourselves to the polarised situation
from the beginning.

\begin{proposition}\label{prop:positive}
Let $E$ be a holomorphic line bundle on a compact complex manifold $X$.
Assume that for each point $x\in X$ there exists a divisor $D\in |E|$ whose complement $X\setminus D$ is a Stein neighbourhood of $x$. Then $E$ is ample. \end{proposition}

\begin{proof}
By the assumption, there are finitely many section $s_0,s_1,\ldots,s_n:X\to E$ 
whose divisors $D_i=\{s_i=0\}$ have empty intersection 
and the domain $X\setminus D_i=\{s_i\neq 0\}$ is Stein for every $i=0,1,\ldots,n$. 
Consider the holomorphic map $\Phi=[s_0:\cdots:s_n]:X\to\CP^n$ \eqref{eq:Phi}. 
We have seen in the proof of Theorem \ref{th:density2} that $E$ is isomorphic 
to the pullback $\Phi^*\Oscr_{\CP^n}(1)$ of the hyperplane section bundle
(cf.\ \cite[Theorem II.7.1]{Hartshorne1977}). 
Given a point $z=[z_0:\cdots:z_n] \in \CP^n$, choose 
$i\in \{0,\ldots,n\}$ such that $z_i\ne 0$ and note that 
$\Phi^{-1}(z)$ is a closed complex subvariety of $X$ contained in 
the Stein domain $X\setminus D_i$. 
Since a Stein manifold does not contain any compact complex subvariety 
of positive dimension, $\Phi^{-1}(z)$ is a finite set (or empty),
so $\Phi$ is a finite holomorphic map. 
It follows that the line bundle $E\cong\Phi^*\Oscr_{\CP^n}(1)$ is ample 
(see Lazarsfeld \cite[proof of Theorem 1.2.13]{Lazarsfeld2004}).
\end{proof}

%
%
\begin{proposition} \label{prop:tensor2}
If a polarised manifold $(X,E)$ has the polarised density property, then so 
does every positive tensor power $(X,E^{\otimes k})$ for $k>0$. 
\end{proposition}

\begin{proof}
If the line bundle $E$ is given on an open cover $\{U_i\}$ of $X$ by a 1-cocycle
$\phi_{i,j}$, then a holomorphic section $f:X\to E$ is given by a collection 
of holomorphic functions $f_i:U_i\to \C$ satisfying $f_i=\phi_{i,j}f_j$ on $U_{i,j}$. 
Since the bundle $E^{\otimes k}$ is given by the 1-cocycle $\phi^k_{i,j}$,
the collection $f_i^k$ defines a holomorphic section 
$f^{\otimes k}$ of $E^{\otimes k}$. Evidently, $\{f=0\}=\{f^{\otimes k}=0\}$. 
By the assumption there are holomorphic sections
$s_0,\ldots,s_n:X\to E$ such that for every $i=0,1,\ldots,n$ 
the domain $X_i=\{s_i\ne 0\}$ is a Stein manifold with the density
property and $\bigcup_{i=0}^n X_i=X$. Hence, for any integer $k\ge 1$ 
the collection $s^{\otimes k}_0,\ldots,s^{\otimes k}_n$ of sections  
of $E^{\otimes k}$ shows that $(X,E^{\otimes k})$ has the polarised density property. 
\end{proof}

%
%
\begin{example}[Line bundles on Grassmannians]\label{ex:Grassmannian}
Given integers $1\le m<n$ we denote by $G_{m,n}$ the 
Grassmann manifold of complex $m$-dimensional subspaces of $\C^n$. 
Note that $G_{1,n}=\CP^{n-1}$. 
These manifolds are complex homogeneous, and hence Oka. 
The Pl\"ucker embedding $P:G_{m,n}\hra \CP^N$, with 
$N={n\choose m}-1$, sends an $m$-plane $\span(v_1,\ldots,v_m)\in G_{m,n}$ 
(where $v_1,\ldots,v_m\in\C^n$ are linearly independent vectors) to 
the complex line in $\C^{N+1}$ given by the vector 
$v_1\wedge \cdots\wedge v_m\in \Lambda^m(\C^n)\cong\C^{N+1}$. 
The intersection of the submanifold $X=P(G_{m,n})\subset \CP^N$ 
with an affine chart $\C^N\cong U\subset \CP^N$ is biholomorphic to
$\C^{m(n-m)}$, which has the density property if $m(n-m)=\dim G_{m,n} >1$.
It follows that the pullback $P^*\Oscr_{\CP^N}(1)$ of the hyperplane section 
bundle on $\CP^N$ to $G_{m,n}$ has the polarised density property and 
Theorem \ref{th:density1} applies to it. Every holomorphic line bundle on 
$G_{m,n}$ is obtained from a line bundle on $\CP^N$, and the pullback map 
$P^*:\Pic(\CP^N)\to \Pic(G_{m,n})$ is a group isomorphism; hence,
$\Pic(G_{m,n})\cong\Z$ (see \cite[Example 1.1.4 (3)]{Brion2005} or 
\cite[Lemma 11.1]{Dolgachev2003}).  The pullback of the universal bundle 
$\U=\Oscr_{\CP^N}(-1)$ is isomorphic to the determinant bundle 
of the universal bundle on $G_{m,n}$, and it generates $\Pic(G_{m,n})$. 
Write $\Oscr_{G_{m,n}}(k)$ for the $(-k)$-th tensor power of this generator. 
Thus, $\Oscr_{G_{m,n}}(1)=P^*\Oscr_{\CP^N}(1)$. 
A line bundle $E\to G_{m,n}$ is positive 
(resp.\ negative) if $E\cong \Oscr_{G_{m,n}}(k)$ for some $k>0$ (resp.\ $k<0$). 
The above observation for $\Oscr_{G_{m,n}}(1)$ and  
Proposition \ref{prop:tensor2} imply the following.

%
%
\begin{proposition}\label{prop:Grassmannian}
Every Grassmannian of dimension $>1$ has the polarised
density property. 
\end{proposition}
\end{example}

To state the next result, consider a pair of polarised manifolds $(X_1,E_1)$ and 
$(X_2,E_2)$. Let $\pi_i:X_1\times X_2\to X_i$ for $i=1,2$ denote the standard 
projections. Then, $\pi_i^*E_i$ is a holomorphic line bundle on $X_1\times X_2$ 
for $i=1,2$. Their tensor product
\[
	E=E_1\boxtimes E_2 := (\pi_1^* E_1)\otimes (\pi_2^* E_2) \to X_1\times X_2
\]
is called the {\em external tensor product} of $E_1$ and $E_2$.
A pair of holomorphic sections $f_i\in H^0(X_i,E_i)$ for $i=1,2$ defines a holomorphic 
section $f_1\boxtimes f_2\in H^0(X_1\times X_2,E_1\boxtimes E_2)$ by
trivially extending both line bundles and sections to 
$X_1\times X_2$ and taking their tensor product.
Similarly, for a pair of semipositive hermitian metrics $h_i$ on $E_i$ for $i=1,2$, 
the semipositive hermitian metric $h=h_1\boxtimes h_2$ on $E_1\boxtimes E_2$ 
is defined in an obvious manner by considering $h_i$ as a hermitian metric 
on $\pi_i^*E_i$. Note that the restriction of $E=E_1\boxtimes E_2$ 
to $X_1\times \{x_2\}$ $(x_2\in X_2)$ is isomorphic to $E_1$, and analogously 
for the second factor. Clearly, this operation extends to any finite number of line
bundles $E_i\to X_i$, $i=1,\ldots,m$. 

%
%
\begin{proposition}\label{prop:product}
If the polarised manifolds $(X_1,E_1)$ and $(X_2,E_2)$ have the polarised density property, then the product $(X_1\times X_2,E_1\boxtimes E_2)$ also has the 
polarised density property.
\end{proposition}

\begin{proof}
Let $f_1,\ldots,f_n\in H^0(X_1,E_1)$ and $g_1,\ldots,g_m\in H^0(X_2,E_2)$ 
be holomorphic sections of the respective line bundles 
which satisfy the definition of the polarised density property.  
As explained above, we may consider both bundles and their section to be 
defined on $X= X_1\times X_2$. Consider the collection of sections 
$f_i\boxtimes g_j\in H^0(X,E_1\boxtimes E_2)$ 
for $i=1,\ldots,n$ and $j=1,\ldots,m$.
For any pair of indices $i,j$ in the given range, the set
\[
	U_{i,j}:=\{f_ig_j\ne 0\} = \{x_1\in X_1:f_i(x_1)\ne 0\} 
	\times \{x_2\in X_2:g_i(x_2)\ne 0\}
\]
is the product of Stein manifolds with the density property, so it is Stein 
with the density property (see Varolin \cite[p.\ 136, I.1]{Varolin2001}).
Since the sets $U_{i,j}$ cover $X$, the proposition follows.
\end{proof}

%
%
\begin{proposition}\label{prop:productCPr}
If the polarised manifold $(X,E)$ has the polarised density property, then 
$(X\times\CP^n, E\boxtimes\Oscr_{\CP^n}(k))$ $(n>0,\ k>0)$ 
also has the polarised density property.
The same is true for $(X \times G_{m,n},E\boxtimes\Oscr_{G_{m,n}}(k))$
with $1\le m<n$ and $k>0$. 
\end{proposition}

\begin{proof}
Since every projective space is also a Grassmannian, it suffices
to consider the second case. If $\dim G_{m,n}>1$, this follows from 
Propositions \ref{prop:Grassmannian} and \ref{prop:product}.
If $\dim G_{m,n}=1$ then $G_{m,n}=\CP^1$. 
We follow the proof of Proposition \ref{prop:product} 
and use that if $X$ is a Stein manifold with the density property 
then $X\times \C$ also has the density property 
(see Varolin \cite[p.\ 136, I.2]{Varolin2001}).
\end{proof}

%
%
\begin{remark}
Assuming that holomorphic line bundles $E_1$ and $E_2$ 
on a projective manifold $X$ have the polarised density property, 
we do not know whether their tensor product 
$E_1\otimes E_2$ has the polarised density property. 
Indeed, given nontrivial sections $f:X\to E_1$ and $g:X\to E_2$, the zero set
of the section $f\otimes g:X\to E_1\otimes E_2$ is $\{f=0\}\cup\{g=0\}$,
and its complement is $\{f\ne 0\}\cap \{g\ne 0\}$.
This manifold need not have the density property even if both 
$\{f\ne 0\}$ and $\{g\ne 0\}$ are Stein manifolds with the density property.
\end{remark}

Recall that a projective manifold is said to be rational if it is birationally isomorphic to a projective space. Every rational curve is isomorphic to $\CP^1$. 
Ischebeck \cite{Ischebeck1974} proved that if $Y$ 
is a rational manifold (in particular, if $Y$ is a projective space or a 
complex Grassmannian) then $\Pic(X\times Y)=\Pic(X)\times \Pic(Y)$, 
so we get all holomorphic line bundles on $X\times Y$ 
as external tensor products of lines bundles on $X$ and $Y$. 

%
%
\begin{proposition}\label{prop:product-rational}
If $X_1,\ldots,X_m$ $(m\ge 2)$ are rational manifolds 
such that every $X_i$ with $\dim X_i>1$ has the polarised 
density property, then their product $X_1\times X_2\times \cdots\times X_m$ 
also has the polarised density property.
\end{proposition}

\begin{proof}
It suffices to prove the result for $m=2$ and apply induction. 
By Ischebeck \cite{Ischebeck1974} we have that 
$\Pic(X)=\Pic(X_1)\times\Pic(X_2)$ and each group $\Pic(X_i)$ is discrete.
Let $E$ be an ample line bundle on $X_1\times X_2$.
Then the restriction of $E$ to each factor $X_i$ $(i=1,2)$ 
is an ample line bundle $E_i$ and $E\cong E_1\boxtimes E_2$. 
If $\dim X_1>1$ and $\dim X_2>1$ then both $E_1$ and $E_2$
have the polarised density property by the assumption,
and the conclusion follows from Proposition \ref{prop:product}. 
If $\dim X_1>1$ and $\dim X_2=1$
then $X_2\cong \CP^1$, $E_2=\Oscr_{\CP^1}(k)$ for some $k>0$, 
and the conclusion follows from Proposition \ref{prop:productCPr}.
The same argument applies if $\dim X_1=1$ and $\dim X_2> 1$. 
In the remaining case, both $X_1$ and $X_1$ are isomorphic to $\CP^1$ 
and $E_i \cong \Oscr_{\CP^1}(k_i)$ for some $k_i>0$ $(i=1,2)$. 
Fixing a point $p=(p_1,p_2)\in X=\CP^1\times \CP^1$ we can find
a pair of holomorphic sections $f_i:\CP^1\to E_i$ $(i=1,2)$ such that 
$p_i\in U_i=\{f_i\ne 0\}\cong\C$. Thus, $f_1f_2$ is a section of
$E\cong E_1\boxtimes E_2$, and the set 
$\{(x_1,x_2)\in X: f_1(x_1) f_2(x_2) \ne 0\}$ is a neighbourhood of $p$
isomorphic to $\C^2$, which has the density property.
This shows that $X$ has the polarised density property.
\end{proof}

Since every complex Grassmannian is a rational manifold, we 
have the following corollary to Propositions 
\ref{prop:Grassmannian} and \ref{prop:product-rational}.

\begin{corollary}\label{cor:productofGrassmannians}
If $X=X_1\times\cdots \times X_m$ is a product of complex Grassmannians
and $\dim X>1$, then $X$ has the polarised density property.
\end{corollary}

We have the following generalization of Proposition \ref{prop:Grassmannian}.
This also implies the polarised density property of any 
hyperquadric \eqref{eq:hyperquadric} of dimension $>1$. 

%
%
\begin{theorem}\label{th:rational}
Every rational homogeneous manifold of dimension $>1$ has the 
polarised density property.
\end{theorem}

\begin{proof}
By the Borel--Remmert theorem \cite{BorelRemmert1962}, 
a rational homogeneous manifold $X$ is a (generalized) flag manifold, 
i.e., $X=G/P$ where $G$ is a semisimple algebraic group and $P$ is a 
parabolic subgroup of $G$. Its Picard group $\Pic(X)$ is discrete and is 
generated by the classes of the Schubert divisors $D_1,\ldots, D_m$, 
and every ample line bundle $E$ on $X$ can be written as
\[
	E\cong[D_1]^{\otimes k_1}\otimes\cdots\otimes[D_m]^{\otimes k_m}
\]
with positive numbers $k_1,\ldots,k_m$. (See Kishimoto et al.\ 
\cite[Sect.\ 1.3]{KishimotoProkhorovZaidenberg2011} and the references 
therein, in particular Brion \cite[Proposition 1.4.1]{Brion2005}. 
In the special case when $P$ is a maximal parabolic subgroup of $G$, 
the Picard group $\Pic(X)\cong\Z$ is generated by a single divisor.
This is the case e.g.\ for Grassmannians, see Example 
\ref{ex:Grassmannian}.)
By the Bruhat decomposition (also called the Schubert decomposition),
the complement of the union of the 
Schubert divisors (which is the unique top dimensional cell 
in the Schubert decomposition of the flag manifold) is isomorphic to 
the affine space of dimension $\dim X>1$, which has the 
density property. (See e.g.\ \cite[Sect.\ 3.3]{LakshmibaiRaghavan2008},
and in particular Remark 3.3.0.2 ibid.)
Thus the support of the divisor
\[
	D=k_1D_1+\cdots+k_mD_m\in|E|
\]
has a Stein complement $X\setminus D$ with the density property.
Since $X$ is homogeneous and the Picard group $\Pic(X)$ is discrete, 
for each point $x\in X$ there exists a holomorphic automorphism 
$\varphi$ of $X$ such that the pullback line bundle $\varphi^* E$ is 
isomorphic to $E$ and the divisor 
$\varphi^*D\in|\varphi^* E|=|E|$ does not contain $x$. Therefore, 
the bundle $E\to X$ satisfies the polarised density property.
\end{proof}

In conclusion, we pose the following open problems.

%
%
\begin{problem}\label{prob:classA} 
Let $X$ be a projective Oka manifold of dimension $>1$ 
and $E$ be an ample holomorphic line bundle on $X$. 
\begin{enumerate}[\rm (a)]
\item Is there a hermitian metric $h$ on $E$ such that the 
disc bundle $\Delta_h(E)$ is an Oka manifold? 
\item Does this hold for every semipositive hermitian metric on $E$? 
\item Does this hold if $X$ is Zariski locally isomorphic to $\C^n$ with $n>1$?
\end{enumerate}
\end{problem}

At present, we do not know any example of a compact complex manifold
of dimension $>1$ with the polarised density property which is not 
rational homogeneous.

%
%
\section{Holomorphic maps from Stein manifolds to vector bundles}
\label{sec:cluster}

Assume that $(E,h)$ is a hermitian holomorphic vector bundle on a compact 
Oka manifold $X$. In this section, we combine the results obtained in this 
paper with those of Drinovec Drnov\v sek and Forstneri\v c \cite{DrinovecForstneric2010AJM} to find holomorphic maps $S\to E$ 
from Stein manifolds $S$ with $\dim S < \dim E$
which are either proper or have their boundary cluster set contained in the 
zero section of $E$. The former case occurs when $(E,h)$ is Griffiths 
negative and the exterior tube 
\begin{equation}\label{eq:Dh}
	D_h(E)=\{e\in E:|e|_h>1\}. 
\end{equation}
is Oka. This holds in particular if $\rank\, E>1$ (see Theorem \ref{th:negative})
or if $(E,h)$ is a negative line bundle on $\CP^n$ 
(see Theorem \ref{th:projective} (b')).
The latter case occurs when $(E,h)$ is a positive line bundle with 
Oka disc bundle $\Delta_h(E)=\{h<1\}$ \eqref{eq:db}; sufficient conditions
are given by Theorems \ref{th:projective}, \ref{th:density1}, 
\ref{th:density2}, and \ref{th:negative}. We begin with the former case.

%
%
\begin{theorem}\label{th:proper}
Let $(E,h)$ be a Griffiths negative hermitian holomorphic vector bundle 
on a compact complex manifold $X$ (see Definition \ref{def:Griffiths-positive}).
Assume that $S$ is a Stein manifold with $\dim S < \dim E$, $K\subset S$
is a compact $\Oscr(S)$-convex subset, and $f_0:S\to E$ is a continuous map
which is holomorphic on a neighbourhood of $K$ and 
satisfies $f_0(S\setminus \mathring K)\subset E\setminus E(0)$.
If the domain $D_h(E)$ \eqref{eq:Dh} is Oka, then we can approximate $f_0$
uniformly on $K$ by proper holomorphic maps $f:S\to E$ homotopic to $f_0$.
Furthermore, if $2\dim S<\dim E$ then $f$ can be chosen an embedding,
and if $2\dim S\le \dim E$ then $f$ can be chosen an immersion.
\end{theorem}

With $(E,h)$ as in the theorem, the domain $D_h(E)$ \eqref{eq:Dh} is Oka 
if $E$ is a line bundle on $X=\CP^n$ (see Theorem \ref{th:projective} (b')), 
or if $\rank\, E>1$ and $X$ is an Oka manifold 
(see Theorem \ref{th:negative}), so the result applies in these 
cases. If $D_h(E)$ is Oka then for every $t>0$ the domain
\begin{equation}\label{eq:Dht}
	D_{h,t}=\{e\in E: |e|_h >t\}
\end{equation}
is Oka as well, since it is biholomorphic to $D_h(E)=D_{h,1}(E)$ by a fibre dilation. 

\begin{proof}
Choose a normal exhaustion 
$B_0\Subset B_1\Subset \cdots \subset \bigcup_{i=0}^\infty B_i=S$ 
by relatively compact, smoothly bounded, strongly pseudoconvex 
domains such that $K\subset B_0$ and the given map $f_0$ is holomorphic on a 
neighbourhood of $\overline B_0$. Also, choose an increasing sequence $0<t_0<t_1<\cdots$ with $\lim_{i\to\infty} t_i=+\infty$. Since the hermitian metric $h$ 
is Griffiths negative, the function 
\begin{equation}\label{eq:phi}
	\phi:E\to[0,+\infty), \quad \phi(e)=|e|_h^2\ \ (e\in E)
\end{equation} 
is strongly plurisubharmonic in $E\setminus E(0)$ (see Proposition 
\ref{prop:positivity-psh}). Clearly, $\phi$ is an exhaustion function on $E$ 
without critical points in $E\setminus E(0)$.

Recall that $\dim S<\dim E$ and 
$f_0(S\setminus \mathring K)\subset E\setminus E(0)$ by the assumption.
By \cite[Theorem 1.1]{DrinovecForstneric2010AJM} we can approximate 
$f_0$ uniformly on $K$ by a holomorphic map $\tilde f_0:\overline B_0\to E$,
which is homotopic to $f_0$ through a family of maps sending 
$\overline B_0\setminus \mathring K$ to $E\setminus E(0)$, 
such that $\tilde f_0(bB_0) \subset D_{h,t_0}(E)$ (see \eqref{eq:Dht}). 
The homotopy condition allows us to extend $\tilde f_0$ to a continuous map 
$\tilde f_0:S\to E$ satisfying $\tilde f_0(S\setminus B_0) \subset D_{h,t_0}(E)$,
and the given homotopy from $f_0$ to $\tilde f_0$ on $\overline B_0$ extends to 
a homotopy between these two maps on all of $S$ sending 
$S\setminus B_0$ to $E\setminus E(0)$. 

Since the tube $D_{h,t_0}(E)$ is biholomorphic to $D_h(E)$, and hence Oka,
we can apply the Oka principle in \cite[Theorem 1.3]{Forstneric2023Indag} 
to approximate $\tilde f_0$ uniformly on $\overline B_0$ by a holomorphic map 
$f_1:S\to E$, homotopic to $\tilde f_0$ by a homotopy as above, 
such that $f_1(S\setminus B_0) \subset D_{h,t_0}(E)$.

We now repeat the same procedure with the map $f_1$. First, we approximate
$f_1$ on $\overline B_0$ by a holomorphic map $\tilde f_1:\overline B_1\to E$
such that $\tilde f_1(\overline B_1\setminus B_0) \subset D_{h,t_0}(E)$ and 
$\tilde f_1(bB_1)\subset D_{h,t_1}(E)$. Next, we extend $\tilde f_1$ to a continuous map
$\tilde f_1:S\to E\setminus E(0)$ 
which agrees with the given holomorphic map $\tilde f_1$
on a neighbourhood of $\overline B_1$ and satisfies 
$\tilde f_1(S\setminus B_1)\subset D_{h,t_1}(E)$.
Since the tube $D_{h,t_1}(E)$ is Oka,
we can apply \cite[Theorem 1.3]{Forstneric2023Indag} 
to approximate $\tilde f_1$ uniformly on $\overline B_1$ by a holomorphic map 
$f_2:S\to E$ such that $f_2(S\setminus B_1) \subset D_{h,t_1}(E)$. 
By the same argument as in the first step,
there is a homotopy connecting $f_1$ to $f_2$ sending $S\setminus \mathring K$
to $E\setminus E(0)$.

Continuing inductively, we find a sequence of holomorphic maps
$f_i: S\to E$ for $i=1,2,\ldots$ such that the following conditions 
hold for every $i\ge 1$:
\begin{enumerate}[\rm (i)] 
\item $f_{i}$ approximates $f_{i-1}$ as closely as desired on $\overline B_{i-1}$. 
\item $f_i(S\setminus B_{i-1}) \subset D_{h,t_{i-1}}(E)$.
\item $f_i$ is homotopic to $f_{i-1}$ through a homotopy sending 
$S\setminus \mathring K$ to $E\setminus E(0)$. 
\end{enumerate}
Assuming as we may that the approximation is close enough at every step,
the sequence $f_i$ converges uniformly on compacts in $S$ to a proper
holomorphic map $f:S\to E$ homotopic to the initial map $f_0$.
(Condition (iii) is only needed to keep the induction going.) 
The additions in the last sentence of the theorem follow by using the 
well-known general position argument. We leave the obvious details to the reader.
\end{proof}

Assuming that $(E,h)$ is a hermitian line bundle on $X$, we have seen in Section 
\ref{sec:prelim} that the tube $D_h(E)$ \eqref{eq:Dht} is fibrewise biholomorphic 
to the punctured disc bundle $\Delta^*_{h^*}(E^*)$ in the hermitian dual bundle 
$(E^*,h^*)$, and the section $E(\infty)$ in the associated $\CP^1$-bundle $\wh E\to X$
corresponds to the zero section $E^*(0)$ of the dual bundle.  
Hence, Theorem \ref{th:proper} implies an analogous result for maps 
$S \to E^* \setminus E^*(0)$ whose cluster set lies in the zero section $E^*(0)$.
However, we can prove a stronger result in this direction, allowing the initial
map $S\to E^*$ to intersect the zero section is a compact set. 
To state the result, we recall the following notion.

A sequence $(x_j)_{j\in \N}$ in a topological space $X$ is said to be divergent 
if for every compact set $K\subset X$ there is $j_0\in \N$ such that 
$x_j\in X\setminus K$ for all $j\ge j_0$. Given a continuous map $f:X\to Y$ 
of topological space with $X$ noncompact, its cluster set is  
\[
	C(f)=\{y\in Y: \text{there is a divergent sequence}\ x_j\in X\ 
	\text{with}\ \lim_{j\to\infty}f(x_j)=y\}.
\]
(If $X$ is compact then $C(f)=\varnothing$.) We have the following result
for maps from Stein manifolds to positive hermitian line bundles on 
compact Oka manifolds.

%
%
\begin{theorem}\label{th:cluster}
Assume that $(E,h)$ is a positive hermitian holomorphic line bundle on a 
compact complex manifold $X$ such that the disc bundle $\Delta_h(E)$ is Oka.
Given a Stein manifold $S$ with $\dim S\le \dim X$, a compact 
$\Oscr(S)$-convex set $K\subset S$, and a continuous map $f_0:S\to E$ 
which is holomorphic on a neighbourhood of $K$ and satisfies 
$f_0(K)\subset \Delta_h(E)$, $f_0$ can be approximated uniformly on $K$ by
holomorphic maps $f:S\to \Delta_h(E)$ homotopic to $f_0$ such that 
$C(f)\subset E(0)$. If in addition $2\dim S\le \dim X$ then $f$ can be chosen 
an injective immersion.
\end{theorem}

\begin{proof}
Since the bundle $(E,h)$ is positive, the function $\sigma=1/h:E\to (0,+\infty)$ 
is strongly plurisubharmonic (see Proposition \ref{prop:psc}). Furthermore, 
$d\sigma \ne 0$ on $E\setminus E(0)$, and for any pair of numbers $0<a<b$ 
the set 
\begin{equation}\label{eq:Eab}
	E_{a,b}=\{e\in E: a \le \sigma(e)\le b\}
\end{equation}
is compact. Let $U\subset S$ be an open Stein domain containing $K$ and
$f_0:S\to E$ be a continuous map which is holomorphic on $U$. Choose a
smoothly bounded strongly pseudoconvex domain $B_0\subset S$
such that $K\subset B_0\subset \bar B_0\subset U$ and $\bar B_0$
is $\Oscr(S)$-convex. Recall that $X$, and hence $E$, are Oka manifolds.
By the transversality theorem for holomorphic maps of Stein manifold
to Oka manifolds (see \cite[Corollary 8.8.7]{Forstneric2017E}), we may assume
that the map $f_0:U\to E$ is transverse to the zero section $E(0)$.
Hence, the set 
\[	
	V_0=\{x\in U: f_0(x)\in E(0)\}
\] 
is a closed complex subvariety 
of $U$ which does not contain any connected component of $U$.
The set $K\cup (\bar B_0\cap V_0)$ is $\Oscr(S)$-convex.
Let $c_0>0$ be chosen such that $f_0(\overline B_0) \subset \{\sigma>c_0\}$. 
Pick numbers $c_1>c_0$ and $\epsilon >0$. 
Choose a compact $\Oscr(S)$-convex set $K'\subset U$ such that 
\begin{equation}\label{eq:Kprime}
	K\cup (\bar B_0\cap V_0) \subset \mathring K'\ \ \ \text{and}\ \ \ 
	\sigma\circ f_0>c_1+1\ \text{on}\ bB_0\cap K'.
\end{equation}
The second condition holds if $K'$ is a sufficiently small neighbourhood
of $K\cup (\bar B_0\cap V_0)$. Set $K_0=K'\cap \bar B_0$. 
We claim that there is a holomorphic map $g:\bar B_0\to E$ 
satisfying the following conditions for a fixed Riemannian distance function 
$\dist$ on $E$:
\begin{enumerate}[\rm (i)]
\item $\dist(g(x),f_0(x))<\epsilon$ for all $x\in K_0$.
\item $\sigma(g(x))>\sigma(f_0(x))-\epsilon$ for all $x\in \bar B_0\setminus K_0$.
\item $\sigma\circ g>c_1$ on $bB_0$.
\item $g$ is homotopic to $f_0$ on $\bar B_0$. 
\end{enumerate}
In the special case when $V_0=\varnothing$ and $K_0$ is a compact 
subset of $D$, a map $g$ with these properties is given by
\cite[Lemma 5.3]{DrinovecForstneric2010AJM}, which is the main 
inductive step in \cite[proof of Theorem 1.1]{DrinovecForstneric2010AJM}.
In the case at hand, the compact set $K_0\subset \bar B_0$ may intersect
$bB_0$, but we have that $\sigma\circ f_0>c_1+1$ on $bB_0\cap K_0$
by condition \eqref{eq:Kprime}. Hence, to ensure condition (iii),
it suffices to apply \cite[Lemma 5.2]{DrinovecForstneric2010AJM} finitely many times for points in the compact set  
$\{x\in bB_0\setminus K_0: \sigma(f_0(x))\le c_1+1\}$.
(The cited lemma amounts to lifting a small piece of $f_0(bB_0)$ to a higher 
level set of $\sigma$ by a prescribed amount, while at the same time
approximating $f_0$ on $K_0$ (condition (i)). 
This lifting procedure uses modifications involving local peak functions 
and gluing, and it is designed in such a way that condition (ii) can be fulfilled. 
Condition (iv) is built into the construction as well.
The fact that the function $\sigma$ is not defined on $E(0)$ is irrelevant in 
this proof since the compact set $K_0$ contains $V_0\cap\bar B_0$ in its 
relative interior, and $\sigma$ is only used on $\bar B_0\setminus K_0$.)

By approximation, we may assume that $g$ is holomorphic on a neighbourhood
of $\bar B_0$, and we can extend it to a continuous map $g:S\to E$
homotopic to $f_0$. We now use the hypothesis that the disc bundle 
$\Delta_h(E)$ is an Oka manifold. Hence, the tube 
\[
	\Omega_{c_1}=E(0)\cup \{\sigma>c_1\}=\Delta_{h,1/c_1}(E)
\]
is Oka as well. By the Oka principle in \cite[Theorem 1.3]{Forstneric2023Indag} 
we can approximate $g$ uniformly on $\overline B_0$ by a holomorphic map 
$f_1:S\to E$, homotopic to $g$, such that 
\[
	f_1(S\setminus B_0) \subset \Omega_{c_1}.
\]
Pick an arbitrary smoothly bounded strongly pseudoconvex domain 
$B_1\subset S$ such that $\overline B_0\subset B_1$ and $\overline B_1$
is $\Oscr(S)$-convex.

Continuing inductively, we obtain a normal exhaustion 
of $S$ by an increasing sequence of smoothly
bounded, strongly pseudoconvex domains 
$B_0\Subset B_1\Subset \cdots \subset \bigcup_{i=0}^\infty B_i=S$,
a sequence of continuous maps $f_i:S\to E$ $(i=0,1,\ldots)$, and an increasing
sequence $0<c_0<c_1<c_2<\cdots$ with $\lim_{i\to\infty}c_i=+\infty$
such that for every $i=1,2,\ldots$ the map $f_i$ is holomorphic on 
$\bar B_i$, it approximates $f_{i-1}$ on $\bar B_{i-1}$, 
and it maps $\overline {B_i\setminus B_{i-1}}$ to 
$\Omega_{c_i}=\Delta_{h,1/c_i}(E)$. Assuming as we may that the approximation
is close enough at every step, the limit map $f=\lim_{i\to\infty}f_i:X\to E$
exists, it approximates $f_0$ on $K$ and is homotopic to it, 
and it satisfies $C(f)\subset E(0)$. We leave the obvious details
of this induction to the reader. If $2\dim S\le \dim X$ then we can additionally
use the general position argument at every step of the induction 
to ensure that the map $f$ is an injective immersion.
\end{proof}

%
%
%
%
\medskip 
\noindent {\bf Acknowledgements.} 
Forstneri\v c is supported by the European Union (ERC Advanced grant HPDR, 101053085) and grants P1-0291, J1-3005, N1-0237 from ARIS, 
Republic of Slovenia. 
He wishes to thank the Erwin Schr\"odinger Institute in Vienna, Austria, 
for the opportunity to present this work at the workshop 
{\em Analysis and Geometry in Several Complex Variables} 
(Vienna, 20--24 November 2023). 
Kusakabe is supported in part by JSPS KAKENHI Grant JP21K20324. 
A part of this work was done during Kusakabe's visit 
to the Institute of Mathematics, Physics and Mechanics in Ljubljana
and the Faculty of Mathematics and Physics, University of Ljubljana.
He wishes to thank both institutions for the hospitality and support. 
The paper was finalised during Forstneri\v c's visit to Kyushu University,
and he wishes to thank it for the hospitality.   
We thank Finnur L\'arusson for a question which led to a 
more natural formulation of Theorem \ref{th:projective},
and Yuji Odaka for suggesting how to prove the polarised density property 
of rational homogeneous manifolds (Theorem \ref{th:rational}) 
by the Bruhat decomposition.

\noindent {\bf Statements and Declarations.}
On behalf of all authors, the corresponding author 
states that there is no conflict of interest. 
The manuscript has no associated data.




%
%

\vspace*{5mm} 
\noindent Franc Forstneri\v c 

\noindent Faculty of Mathematics and Physics, University of Ljubljana, Jadranska 19, SI--1000 Ljubljana, Slovenia 

\noindent 
Institute of Mathematics, Physics and Mechanics, Jadranska 19, SI--1000 Ljubljana, Slovenia 

\noindent e-mail: {\tt franc.forstneric@fmf.uni-lj.si} 

\vspace*{5mm} 
\noindent Yuta Kusakabe 

\noindent Faculty of Mathematics, Kyushu University,
744 Motooka, Nishi-ku, Fukuoka 819-0395, Japan

\noindent e-mail: {\tt kusakabe@math.kyushu-u.ac.jp} 

\end{document}